\documentclass[9.8pt,letterpaper]{amsart}
\usepackage[utf8]{inputenc}
\usepackage{amsmath}
\usepackage{amsfonts}
\usepackage{amssymb}
\usepackage{amsthm}
\usepackage{chngcntr}
\usepackage{mathtools}
\usepackage[all,cmtip]{xy}
\usepackage{tikz}
\usepackage[pagebackref=false,colorlinks]{hyperref}
\hypersetup{pdffitwindow=true,linkcolor=blue,citecolor=blue,urlcolor=cyan}
\usepackage{cite}
\usepackage{fancyhdr}
\usepackage{stmaryrd}
\usepackage{yfonts}
\usepackage{cases}

\font\wncyr=wncyr9.8
\newcommand{\sha}{\text{\wncyr{W}}}

\usepackage[top=2.6cm, bottom=2.4cm, left=2.4cm, right=2.4cm]{geometry}

\usepackage{tikz}
\usepackage{tikz-cd}

\usepackage{hyperref}

\newtheorem{theorem}{Theorem}
\newtheorem{lemma}[theorem]{Lemma}
\newtheorem{prop}[theorem]{Proposition}
\newtheorem{cor}[theorem]{Corollary}
\newtheorem{conv}[theorem]{Convention}
\newtheorem{nota}[theorem]{Notation}

\theoremstyle{definition}
\newtheorem{definition}[theorem]{Definition}
\newtheorem{remark}[theorem]{Remark}

\newtheorem{theoremintro}{Theorem}

\newcommand{\xrightarrowdbl}[2][]{%
  \xrightarrow[#1]{#2}\mathrel{\mkern-14mu}\rightarrow}
\newcommand{\rank}{\operatorname{rank}}
\newcommand{\corank}{\operatorname{corank}}

\renewcommand{\thetheorem}{\arabic{section}.\arabic{subsection}.\arabic{theorem}}

\makeatletter
\@addtoreset{theorem}{section}
\makeatother

\begin{document}

\title{Structure of ordinary $\Lambda$-adic arithmetic cohomology groups}
\author{JAEHOON LEE}
\maketitle

\begin{abstract} We study the $\Lambda$-module structure of the ordinary parts of the arithmetic cohomology groups of modular Jacobians made out of various towers of modular curves. We prove that the ordinary parts of $\Lambda$-adic Selmer groups coming from two different towers have ``almost same" $\Lambda$-module structures. We also prove the cotorsionness of $\Lambda$-adic Tate-Shafarevich group under mild assumptions.
\end{abstract}

\tableofcontents

\section{Introduction}

\smallskip

\subsection{Basic Setups}

\medskip 

Fix one rational prime $p \geq 5$, a positive integer $N$ prime to $p$ and $W$ an integer ring of a finite unramified extension of $\mathbb{Q}_p$ throughout the paper. We let $\Lambda:=W[\![T]\!]$, the one variable power series ring over $W$. We also define $\omega_n=\omega_n(T):=(1+T)^{p^{n-1}}-1$. 

\medskip

In this paper, we study the various modular curves defined in \cite{hida2015analytic}. (See Definition \ref{definition 2.2.2} for details) For simplicity, in this introduction 
let $X_{r /\mathbb{Q}}$ be the (compactified) modular curve classifying the triples 
$$(E, \mu_N \xhookrightarrow{\phi_N} E, \mu_{p^r} \xhookrightarrow{\phi_{p^r}} E[p^r] )_{/R}$$ where $E$ is an elliptic curve defined over a $\mathbb{Q}$-algebra $R$. Define $X'_{r /\mathbb{Q}}$ be the another modular curve classifying the triples 
$$(E, \mu_N \xhookrightarrow{\phi_N} E, E[p^r] \xrightarrowdbl{\varphi_{p^r}} \mathbb{Z}/p^r\mathbb{Z})_{/R}.$$ They are smooth projective curves defined over $\mathbb{Q}$ on which the group $(\mathbb{Z}/N\mathbb{Z})^{\times}\times\mathbb{Z}_p^{\times}$ acts naturally. (For an explicit description, see subsection \ref{sec 2.2})

\begin{remark}\label{remark 1.1.1} As explained in the introduction of \cite{hida2015analytic}, the tower of modular curves $ \lbrace X_{r /\mathbb{Q}} \rbrace$ (resp. $ \lbrace X'_{r /\mathbb{Q}} \rbrace$) corresponds to one variable $p$-adic L-function interpolating $L(k+2, k+1)$ (resp. $L(k+2, 1)$) for integers $k$, where $L$ is a Kitagawa-Mazur two variable $p$-adic $L$-function. Studying the general tower $ \lbrace X_{\alpha, \delta, \xi}(Np^r)_{/\mathbb{Q}} \rbrace$ (see subsection \ref{sec 2.2} for the precise definition) and moving $(\alpha, \delta)$ corresponds to studying various one-variable specialization of $L$-function (interpolating $L((\alpha+\delta)k+2, \delta k+1)$ for integers $k$) on analytic side.
\end{remark}

Let $J_{r /\mathbb{Q}}$ and $J_{r /\mathbb{Q}}'$ be the Jacobians of $X_r$ and $X'_r$, respectively. For a number field $K$, we have the Mordell-Weil group $J_{r}(K) \otimes \mathbb{Q}_p/\mathbb{Z}_p$, (geometric) $p$-adic Selmer group $\mathrm{Sel}_{K}(J_{r})_p$ and the $p$-adic Tate-Shafarevich group $\sha^{1}_{K}(J_{r})_p$ of the Jacobian $J_r$ (See \cite[Page 74, 75]{milne2006arithmetic} for the definition). We can apply the idempotent $\displaystyle e:=\lim_{n \rightarrow \infty}U(p)^{n!}$ to those groups to take ordinary parts $\left(J_{r}(K) \otimes \mathbb{Q}_p/\mathbb{Z}_p\right)^{ord}$, $\mathrm{Sel}_{K}(J_{r})^{ord}_p$ and $\sha^{1}_{K}(J_{r})^{ord}_p$. (See Definition \ref{definition 4.1.3} for the precise definition.) By the Picard functoriality, we can consider \begin{itemize}
\item $\displaystyle \left(J_{\infty}(K) \otimes \mathbb{Q}_p/\mathbb{Z}_p\right)^{ord}:=\lim_{\substack{\longrightarrow \\ r}}\left(J_{r}(K) \otimes \mathbb{Q}_p/\mathbb{Z}_p\right)^{ord}$
\item $\displaystyle \mathrm{Sel}_{K}(J_{\infty})^{ord}_p:=\lim_{\substack{\longrightarrow \\ r}}\mathrm{Sel}_{K}(J_{r})^{ord}_p$
\item $\displaystyle \sha^{1}_{K}(J_{\infty})^{ord}_p:=\lim_{\substack{\longrightarrow \\ r}}\sha^{1}_{K}(J_{r})^{ord}_p$
\end{itemize} which are naturally $\Lambda$-modules on which we will give discrete topology. We call them as the $\Lambda$-adic Mordell-Weil group, $\Lambda$-adic Selmer group and the $\Lambda$-adic Tate-Shafarevich group in order following \cite{hida2015analytic}. 

\smallskip

For the Selmer groups, we have the natural map $$s_r:\mathrm{Sel}_{K}(J_{r})^{ord}_p \rightarrow \mathrm{Sel}_{K}(J_{\infty})^{ord}_p[\omega_{r}]$$ induced by the restrictions whose kernel $\mathrm{Ker}(s_r)$ is finite and bounded independent of $n$. (See Lemma \ref{lemma 7.0.2} and Remark \ref{remark 8.1.3}-(3)) For the cokernel, we have the following theorem of Hida  (\cite[Theorem 10.4]{hida2015analytic}) which we refer as ``Hida's control theorem" in this article.

\smallskip

\begin{theorem}[Hida]\label{Theorem H}  If $X_r$ does not have split multiplicative reduction at all places of $K$ dividing $p$, then $\mathrm{Coker}(s_r)$ is finite. In particular, $\mathrm{Coker}(s_r)$ is finite for all $r \geq 2$.
\end{theorem}

\smallskip

Hence the behavior of (the \emph{ordinary} parts of) the Selmer groups made out of towers of modular curves are \emph{controlled} by the one $\Lambda$-adic object $\mathrm{Sel}_{K}(J_{\infty})^{ord}_p$ and it is  natural to study the \emph{structure} of $\mathrm{Sel}_{K}(J_{\infty})^{ord}_p$ as a $\Lambda$-module (under the control theorem).

\smallskip

For a finitely generated $\Lambda$-module $M$, there is a $\Lambda$-linear map $$\displaystyle M \rightarrow \Lambda^{r}\oplus \left(\bigoplus_{i=1}^{n}\frac{\Lambda}{g_{i}^{e_i}}\right)\oplus \left( \bigoplus_{j=1}^{m}\frac{\Lambda}{p^{f_j}} \right)$$ with finite kernel and cokernel where $r, n, m \geq 0$, $e_1, \cdots e_n, f_1, \cdots f_m$ are positive integers, and $g_1, \cdots g_n$ are distinguished irreducible polynomial of $\Lambda$. The quantities $r, e_1, \cdots e_n, f_1, \cdots f_m, g_1, \cdots g_n$ are uniquely determined, and we call $$\displaystyle E(M):=\Lambda^{r}\oplus \left(\bigoplus_{i=1}^{n}\frac{\Lambda}{g_{i}^{e_i}}\right)\oplus \left( \bigoplus_{j=1}^{m}\frac{\Lambda}{p^{f_j}} \right)$$ as an \emph{elementary module of $M$} following \cite[Page 292]{neukirch2000cohomology}. \\

In this paper, we want to study the module $$E\left( (\mathrm{Sel}_{K}(J_{\infty})^{ord}_p)^{\vee} \right).$$ We will also consider the elementary modules of (the Pontryagin dual of) $$\left(J_{\infty}(K) \otimes \mathbb{Q}_p/\mathbb{Z}_p\right)^{ord} \quad \mathrm{and} \quad \sha^{1}_{K}(J_{\infty})^{ord}_p.$$

\smallskip

\begin{remark}\label{remark 1.1.3} Hida \cite[Theorem 5.6]{hida2015analytic} proved that in most cases, the connected component of big ordinary Hecke algebra is isomorphic to $\Lambda$. Hence our result also describes Hecke module structure of cohomology groups for those cases.
\end{remark}

\begin{remark}\label{remark 1.1.4}
\emph{The main novelties} of this paper are:
\begin{itemize}
\item This paper has three main results having similar flavour (and proofs) with those of \cite{Lee2018}. The main difference is the ``origin" of $p$-adic variation: including \cite{Lee2018}, the Selmer groups of an abelian variety (or Galois representations) under \emph{cyclotomic} variations are usually considered. (For instance, \cite{mazur1972rational}) On the other hand, this paper deals with the $p$-adic variation of the \emph{classical} Selmer groups from \emph{towers of modular curves.}
\smallskip
\item The main results of this paper works under Hida's control theorem and does not need the \emph{cotorsionness} assumption on the Selmer groups.
\smallskip
\item Instead of the characteristic ideals of the modules above (which are usually studied because of their connections with the Iwasawa Main Conjecture), we can study the $\Lambda$-module \emph{structure} of the Mordell-Weil, Selmer and Tate-Shafarevich groups in a purely algebraic way.
\end{itemize}
\end{remark}

\smallskip

\subsection{Statements of the main results} \label{sec 1.2}

Our first result compares the structure of the Selmer groups for two towers $ \lbrace X_{r /\mathbb{Q}} \rbrace$ and $ \lbrace X'_{r /\mathbb{Q}} \rbrace$ as $\Lambda$-modules. On the analytic side, we have a functional equation of the $p$-adic $L$-function relating two values $L(k+2, k+1)$ and $L(k+2, 1)$. (See \cite[Corollary 17.2]{mazur1986onp} and \cite[Theorem 3.8.1-(2)]{greenberg1994conjecture})
From the perspective of the Iwasawa Main Conjecture, we can expect that the similar duality result would be true on the algebraic side also. 

\smallskip

\begin{theoremintro}[Main Theorem $\mathrm{I}$: Theorem \ref{theorem 8.3.10}]\label{Theorem A} We have an isomorphism $$E\left( (\mathrm{Sel}_{K}(J_{\infty})^{ord}_p)^{\vee} \right) \simeq E\left( (\mathrm{Sel}_{K}(J'_{\infty})^{ord}_p)^{\vee} \right)^{\iota}$$ of $\Lambda$-modules. Here $\iota$ is an involution of $\Lambda$ satisfying $\iota(T)=\frac{1}{1+T}-1$.
\end{theoremintro}

\smallskip

\begin{remark}\label{remark 1.2.5} (1) This result is not just an equality of characteristic ideals but a statement about the \emph{isomorphism classes}.

\smallskip

(2) If $\mathrm{Sel}_{K}(J_{\infty})^{ord}_p$ and $\mathrm{Sel}_{K}(J'_{\infty})^{ord}_p$ are cotorsion $\Lambda$-modules, let $g(T)$ (resp. $h(T)$) be a generator of $char_{\Lambda}(\mathrm{Sel}_{K}(J_{\infty})^{ord}_p)^{\vee}$ (resp. $char_{\Lambda}(\mathrm{Sel}_{K}(J'_{\infty})^{ord}_p)^{\vee}$) for the moment. They satisfy a relation $$g(T)=h(\frac{1}{1+T}-1)$$ up to unit in $\Lambda$, which can be regarded as an algebraic counterpart of \cite[Theorem 3.8.1-(2)]{greenberg1994conjecture}.

\smallskip

(3) In particular, for the self-dual line case in the sense of \cite[Page 1]{hida2015analytic} (which is expected to be the case of $p$-adic $L$-function interpolating $L(2k+2, k+1)$ on the analytic side), we get $$g(T)=g(\frac{1}{1+T}-1)u(T)$$ for some $u(T) \in \Lambda^{\times}$. We expect that $u(0)=(-1)^{\text{ord}_{T=0}(g(T))}$ (which is well-defined) is the root number of the $p$-adic $L$-function (For instance, \cite[Section 18]{mazur1986onp}) on the analytic side.

\smallskip

(4) For the proof of this theorem, in technical aspect, we construct a \emph{twisted} version of the perfect pairing of Flach (\cite{flach1990generalisation}) between Selmer groups of an abelian variety $A$ over a number field. Here the word \emph{twisted} means twisting the original pairing by the Weil involution $w_{r /\mathbb{Q}}$ inducing an isomorphism $$w_{r /\mathbb{Q}}:J_{r} \xrightarrow{\sim} J_{r}'.$$ (For instance, see the first part of \ref{sec 5.1}) This allows us to have a perfect pairing between the \emph{ordinary} parts of the Selmer groups. (Proposition \ref{prop 5.4.11}) By the functorial property of our pairing, we can lift this duality to the $\Lambda$-adic setting (by using two functors $\mathfrak{F}$ and $\mathfrak{G}$ defined in subsection \ref{sec 6.2}), which gives the result.
\end{remark}

\medskip

To analyze the $\Lambda$-module structure of $\sha^{1}_{K}(J_{\infty})^{ord}_p$, we need to introduce one functor defined in \cite[Appendix]{Lee2018}. For a finitely generated $\Lambda$-module $X$, we define $\displaystyle \mathfrak{G}(X):=\lim_{\substack{\longleftarrow \\n}}\left(\frac{X}{\omega_nX}[p^{\infty}] \right)$ 
where the inverse limit is taken with respect to the natural projections. (For an explicit description of this functor, see subection \ref{sec 6.2}.) Since $\mathfrak{G}(\Lambda)=0$, $\mathfrak{G}(X)$ is a finitely generated torsion $\Lambda$-module. The following statement can be regarded as a $\Lambda$-adic analogue of the Tate-Shafarevich conjecture. 
\medskip

\begin{theoremintro}[Main Theorem $\mathrm{II}$: Theorem \ref{theorem 9.2.5}]\label{Theorem B} Suppose that $\sha^{1}_{K}(J_{r})^{ord}_p$ is finite for all $r$. Then the functor $\mathfrak{G}$ induces an isomorphism $$\left(\sha^{1}_{K}(J_{\infty})^{ord}_p\right)^{\vee} \simeq \mathfrak{G}\left((\mathrm{Sel}_{K}(J_{\infty})^{ord}_p)^{\vee}\right)$$ of $\Lambda$-modules. In particular, $\sha^{1}_{K}(J_{\infty})^{ord}_p$ is a cofinitely generated cotorsion $\Lambda$-module.
\end{theoremintro}

\smallskip

\begin{remark}\label{remark 1.2.6} (1) As we remarked earlier, this result is a $\Lambda$-adic analogue of the Tate-Shafarevich conjecture in a sense that the $\mathbb{Z}_p$-cotorsionness of $\sha^{1}_{K}(J_{r})^{ord}_p$ can be lifted to the $\Lambda$-cotorsionness of $\sha^{1}_{K}(J_{\infty})^{ord}_p$.

\smallskip

(2) This result \emph{distinguishes} the Mordell-Weil group and the Tate-Shafarevich group from the Selmer group. More precisely, if we know the elementary module of $(\mathrm{Sel}_{K}(J_{\infty})^{ord}_p)^{\vee}$, then we can explicitly describe elementary modules of (Pontryagin duals of) $$\left(J_{\infty}(K) \otimes \mathbb{Q}_p/\mathbb{Z}_p\right)^{ord} \quad \mathrm{and} \quad \sha^{1}_{K}(J_{\infty})^{ord}_p,$$ since we have an explicit description of the functor $\mathfrak{G}$ (Proposition \ref{prop 6.2.4}).

\smallskip

(3) If we assume the Tate-Shafarevich conjecture for dual tower also (i.e. assuming the same conditions for $J_{r}'$), combining Theorem \ref{Theorem A} and Theorem \ref{Theorem B} shows that the elementary modules of $\left(\sha^{1}_{K}(J_{\infty})^{ord}_p\right)^{\vee}$ and $\left(\sha^{1}_{K}(J'_{\infty})^{ord}_p\right)^{\vee}$ are isomorphic up to twisting by $\iota$. (See Corollary \ref{cor 9.2.6} for the precise statement.)
\end{remark}

\medskip

Assuming the boundedness of $\mathrm{Coker}(s_r)$ additionally, we can explicitly compute an estimate of the size of $\sha^{1}_{K}(J_{r})^{ord}_p$. 

\smallskip

\begin{theoremintro}[Main Theorem $\mathrm{III}$: Theorem \ref{theorem 9.2.7}]\label{Theorem C} Suppose that the group $\mathrm{Coker}(s_r)$ is finite and bounded independent of $r$, and also suppose that $\sha^{1}_{K}(J_{r})^{ord}_p$ is finite for all $r$. Then there exists an integer $\nu$ independent of $n$ such that $$ | \sha^{1}_{K}(J_{r})^{ord}_p |=p^{e_n} \quad (n>>0).$$ where $$e_n=p^n \mu \left((\sha^{1}_{K}(J_{\infty})^{ord}_p)^{\vee} \right) + n \lambda \left((\sha^{1}_{K}(J_{\infty})^{ord}_p)^{\vee} \right) +\nu.$$
\end{theoremintro}

\smallskip

\begin{remark}\label{remark 1.2.7} This theorem is an evidence for the control of the Tate-Shafarevich groups over tower of modular curves. More precisely, if the characteristic ideal of $(\sha^{1}_{K}(J_{\infty})^{ord}_p)^{\vee}$ is coprime to $\omega_n$ for all $n$, then by the $\Lambda$-module theory, we get $$ | \sha^{1}_{K}(J_{\infty})^{ord}_p [\omega_n]|=p^{e_n}.$$ for the same $e_n$ appearing in the above theorem. Hence we can expect that the natural map $$\sha^{1}_{K}(J_{r})^{ord}_p \rightarrow \sha^{1}_{K}(J_{\infty})^{ord}_p[\omega_n]$$ to have finite and bounded kernel and cokernel.
\end{remark}

\smallskip

\subsection{Organization of the paper} \label{sec 1.3}

\begin{itemize}
\item We record some lemmas from the commutative algebra in the appendix.
\item This paper can be divided into four parts. For the first part (Section 2, 3, 4), following \cite{hida2015analytic} we recall the basic settings of paper. In the Section 2, we define various modular curves. We define and briefly recall important facts about the big ordinary Hecke algebra in Section 3. In Section 4, we define our main objects: $\Lambda$-adic Mordell-Weil groups and Barsotti-Tate groups (Definition \ref{definition 4.2.5}) and recall the control results from \cite{hida2015analytic}. 
\medskip
\item The second part is Section 5 which is technically important throughout the paper. We construct various twisted perfect pairings which are crucial in comparing two different towers of modular curves. After twisting by the Weil involution, the original (Weil, Poitou-Tate, Flach) pairings become Hecke-equivariant which allow us to compare the ordinary parts of two arithmetic cohomology groups. (Proposition \ref{prop 5.1.5}, \ref{prop 5.3.9} and \ref{prop 5.4.11})
\medskip
\item For the third part, in Section 6 and 7, we analyze the structure of local and global cohomology groups of the Barsotti-Tate groups $\mathfrak{g}$ (Definition \ref{definition 4.2.5}) as $\Lambda$-modules. We use functors $\mathfrak{F}$ and $\mathfrak{G}$ defined in \cite{Lee2018}. The explicit description of two functors $\mathfrak{F}$ and $\mathfrak{G}$ will be given in Proposition \ref{prop 6.2.4}.
\medskip
\item The last part is Section 8 and 9. In Section 8, we study the $\Lambda$-adic Selmer groups. In \ref{sec 8.2}, we prove the surjectivity of the ``Selmer-defining map" (Proposition \ref{prop 8.2.7} and Corollary \ref{cor 8.2.8}) which is an analogue of \cite[Proposition 2.1]{greenberg2000iwasawa}.
In the subsection \ref{sec 8.3}, we prove one of the main results of this paper (Theorem \ref{theorem 8.3.10}). 
\medskip
\item In Section 9, we will analyze the $\Lambda$-module structure of the $\Lambda$-adic Mordell-Weil group and the $\Lambda$-adic Tate-Shafarevich groups. We show another main result mentioned in introduction (Theorem \ref{theorem 9.2.5}) and give an estimate on the size of an ordinary part of Tate-Shafarevich group of the modular Jacobians. (Theorem \ref{theorem 9.2.7})
\end{itemize}

\medskip

\subsection{Notations}\label{sec 1.4} We fix notations below throughout the paper.

\smallskip

\begin{itemize}
\item Notation \ref{No}, Convention \ref{Id} will be used from Section 5. Notation \ref{Selmer} will also be used from Section 8.
\medskip
\item We fix one rational prime $p \geq 5$ and a positive integer $N$ prime to $p$.
\medskip
\item Also fix $W$, which is a $p$-adic integer ring of a finite unramified extension of $\mathbb{Q}_p$. Note that $p$ is a uniformizer of $W$. We let $Q(W)$ be the quotient field of $W$.
\medskip
\item We let $\Lambda:=W[\![T]\!]$, the one variable power series ring over $W$. We also define $\omega_n=\omega_n(T):=(1+T)^{p^{n-1}}-1$ and $\omega_{n+1, n}:=\frac{\omega_{n+1}(T)}{\omega_n(T)}$. Note that $\omega_{n+1, n}$ is a distinguished irreducible polynomial in $W[\![T]\!]$.
\medskip
\item For a locally compact Hausdorff continuous $\Lambda$-module $M$, we define $M^{\vee}:=\text{Hom}_{cts}(M, Q(W)/W)$ which is also a locally compact Hausdorff. $M^{\vee}$ becomes a continuous $\Lambda$-module via action defined by $(f\cdot \phi)(m):=\phi(\iota(f)\cdot m)$ where $f \in \Lambda$, $m \in M$, $\phi \in M^{\vee}$. We also define $M^{\iota}$ to be same underlying set $M$ whose $\Lambda$-action is twisted by an involution $\iota
:T \rightarrow \frac{1}{1+T}-1$ of $W[\![T]\!]$.
\medskip
\item If $K$ is a number field and $v$ is a place of $K$, we let $K_v$ be the completion of $K$ at $v$. If $S$ is a finite set of places of $K$ containing all infinite places and places over $Np$, we let $K^{S}$ be a maximal extension of $K$ unramified outside $S$.  We remark here that since prime $p$ is odd, $\text{Gal}(K^S/K)$ has $p$-cohomological dimension 2. When we deal with abelian varieties, we suppose that this finite set $S$ contains infinite places, places over $p$ and primes of bad reduction of an abelian variety.
\medskip
\item For an integral domain $R$ and an $R$-module $M$, we define $M_{R-\mathrm{tor}}$ as the maximal $R$-torsion submodule of $M$. Also we define $\rank_{R}M:=\dim_{Q(R)}\left(M \otimes_{R}Q(R) \right)$ where $Q(R)$ is a quotient field of $R$. 
\medskip 
\item For a cofinitely generated $\mathbb{Z}_p$-module $X$, we define $X_{/ div}:=\frac{X}{X_{div}}$ where $X_{div}$ is the maximal $p$-divisible subgroup of $X$. Note that $\displaystyle X_{/ div} \simeq \displaystyle \lim_{\substack{\longleftarrow \\ n}}\frac{X}{p^nX}$ and $\left( X_{/div} \right)^{\vee} \simeq X^{\vee}[p^{\infty}]$.
\medskip
\item For a finitely generated $\Lambda$-module $M$, there are prime elements $g_1, \cdots g_n$ of $\Lambda$, non-negative integer $r$, positive integers $e_1, \cdots e_n$ and a pseudo-isomorphism $\displaystyle M \rightarrow \Lambda^{r}\oplus \left(\bigoplus_{i=1}^{n}\frac{\Lambda}{g_{i}^{e_i}} \right) $. We call $$\displaystyle E(M):=\Lambda^{r}\oplus \left(\bigoplus_{i=1}^{n}\frac{\Lambda}{g_{i}^{e_i}} \right)$$ as an \emph{elementary module} of $M$ following \cite[Page 292]{neukirch2000cohomology}. 
\end{itemize}

\subsection{Acknowledgements} The author would like to thank Haruzo Hida for his endless guidance and encouragement. The author also thanks to Chan-Ho Kim for his comments and suggestions on the first draft of this paper, and KIAS for the hospitality during the Summer 2018.

\medskip

\section{Various modular curves}

We define and analyze some properties of our main geometric objects in this section.

\subsection{Congruence Subgroups} \label{sec 2.1}

For $(\alpha, \delta) \in \mathbb{Z}_p \times \mathbb{Z}_p $ satisfying $\alpha\mathbb{Z}_p+ \delta \mathbb{Z}_p=\mathbb{Z}_p$, define a map $$\pi_{\alpha, \delta}:\Gamma \times \Gamma \rightarrow \Gamma$$ by $$\pi_{\alpha, \delta}(x, y)=x^{\alpha}y^{-\delta}$$ where $\Gamma=1+p\mathbb{Z}_p$ and let $\Lambda_{\alpha, \delta}:= W[\![ (\Gamma \times \Gamma) / \text{Ker}(\pi_{\alpha, \delta}) ]\!]$ be the completed group ring. 

\smallskip

Let $\mu $ be the maximal torsion subgroup of $\mathbb{Z}_p^{\times}$. Pick a character $\xi: \mu \times \mu \rightarrow \mu$ and define another character $\xi': \mu \times \mu \rightarrow \mu$ by $\xi'(a, b):=\xi(b, a)$. Consider $\text{Ker}(\xi) \times \text{Ker}(\pi_{\alpha, \delta}) $ as a subgroup of $\mathbb{Z}_p^{\times}\times \mathbb{Z}_p^{\times}$ and let $H_{\alpha, \delta, \xi, r}$ be the image of $\text{Ker}(\xi) \times \text{Ker}(\pi_{\alpha, \delta}) $ in $(\mathbb{Z}_p^{\times})^{2} / (\Gamma^{p^{r-1}})^{2} $.\\

\begin{definition}\label{definition 2.1.1} Let $\displaystyle \hat{\mathbb{Z}}=\prod_{l:\text{prime}}\mathbb{Z}_l$ and $N$ be a positive integer prime to $p$. For a positive integer $r$, we define the following congruence subgroups:

(1) $\widehat{\Gamma}_{0}(Np^r):=\left\lbrace 
\begin{pmatrix}
a & b \\
c & d
\end{pmatrix} 
 \in \text{GL}_{2}(\widehat{\mathbb{Z}})\mid c \in Np^r\widehat{\mathbb{Z}} \right\rbrace$. Let $ \Gamma_{0}(Np^n)=\widehat{\Gamma}_{0}(Np^n) \cap \text{SL}_2(\mathbb{Q})$.

(2) $\widehat{\Gamma}_{\alpha, \delta, \xi}(Np^r):=\left\lbrace 
\begin{pmatrix}
a & b \\
c & d
\end{pmatrix} 
 \in \widehat{\Gamma}_{0}(Np^r)\mid (a_p, d_p) \in H_{\alpha,\delta, \xi, r}, d-1 \in N\widehat{\mathbb{Z}} \right\rbrace$. Let $$\Gamma_{\alpha, \delta, \xi}(Np^n)=\widehat{\Gamma}_{\alpha, \delta, \xi}(Np^n)\cap \text{SL}_2(\mathbb{Q}).$$

\smallskip

(3) $\widehat{\Gamma}_{1}^{1}(Np^r):=\left\lbrace 
\begin{pmatrix}
a & b \\
c & d
\end{pmatrix} 
 \in \widehat{\Gamma}_{0}(Np^r)\mid c, d-1 \in Np^r\widehat{\mathbb{Z}}, a-1 \in p^{r}\widehat{\mathbb{Z}} \right\rbrace$.
\end{definition}

\smallskip

\begin{remark}\label{remark 2.1.2} We have an isomorphism $$\displaystyle (\mathbb{Z}/N\mathbb{Z})^{\times}\times\mathbb{Z}_p^{\times} \times\mathbb{Z}_p^{\times} \simeq \lim_{\substack{\longleftarrow \\ r}}\left(\widehat{\Gamma}_{0}(Np^r)/\widehat{\Gamma}_{1}^{1}(Np^r)\right)$$ sending $(u, a, d)$ to the matrix  whose $p$-component is $\begin{pmatrix}
a & 0 \\
0 & d
\end{pmatrix}$, $l$-component is $\begin{pmatrix}
1 & 0 \\
0 & u
\end{pmatrix}$ for $l \mid N$ and trivial for $l \nmid Np$. This isomorphism induces another isomorphisms
\begin{align*}
\text{Ker}(\xi) \times H_{\alpha, \delta, \xi, r} &\simeq \widehat{\Gamma}_{\alpha, \delta, \xi}(Np^r) /\widehat{\Gamma}_{1}^{1}(Np^r),\\
(\mathbb{Z}/N\mathbb{Z})^{\times} \times \text{Im}(\xi) \times \Gamma / \Gamma^{p^{r-1}} &\simeq  \widehat{\Gamma}_{0}(Np^r) /\widehat{\Gamma}_{\alpha, \delta, \xi}(Np^r).
\end{align*}
\end{remark}

\medskip

\subsection{Modular Curves} \label{sec 2.2}

Following \cite[Section 3]{hida2015analytic}, we study various modular curves arose from congruence subgroups we defined in the previous subsection. Consider the moduli problem over $\mathbb{Q}$ classifying the triples $$(E, \mu_N \xhookrightarrow{\phi_N} E, \mu_{p^r} \xhookrightarrow{\phi_{p^r}} E[p^r] \xrightarrowdbl{\varphi_{p^r}} \mathbb{Z}/p^r\mathbb{Z})$$ where $E$ is an elliptic curve defined over $\mathbb{Q}$-algebra $R$ and the sequence $$\mu_{p^r} \xhookrightarrow{\phi_{p^r}} E[p^r]\xrightarrowdbl{\varphi_{p^r}} \mathbb{Z}/p^r\mathbb{Z}$$ is exact in the category of finite flat group schemes. The triples are classified by a modular curve $U_{r / \mathbb{Q}}$ and we let $X_{1}^{1}(Np^r)_{/\mathbb{Q}}$ be the compactification of $U_r$ smooth at cusps. \\

Note that $(\mathbb{Z}/N\mathbb{Z})^{\times}\times\mathbb{Z}_p^{\times}\times\mathbb{Z}_p^{\times}$ naturally acts on $ X_{1}^{1}(Np^r)_{/\mathbb{Q}}$ : $ (u, a, d) \in (\mathbb{Z}/N\mathbb{Z})^{\times}\times \mathbb{Z}_p^{\times} \times\mathbb{Z}_p^{\times}$ sends $$ (E, \mu_N \xhookrightarrow{\phi_N} E, \mu_{p^r} \xhookrightarrow{\phi_{p^r}} E[p^r] \xrightarrowdbl{\varphi_{p^r}} \mathbb{Z}/p^r\mathbb{Z})$$ to $$(E, \mu_N \xhookrightarrow{\phi_N \circ u} E, \mu_{p^r} \xhookrightarrow{\phi_{p^r} \circ a} E[p^r] \xrightarrowdbl{d \circ \varphi_{p^r}} \mathbb{Z}/p^r\mathbb{Z}).$$

\medskip

We now define our main geometric objects.

\begin{definition}\label{definition 2.2.3} We define the following two curves:
\begin{align*}
X_{\alpha, \delta, \xi}(Np^r)_{/ \mathbb{Q}}&:=X_{1}^{1}(Np^r)_{/ \mathbb{Q}}/\left(\mathrm{Ker}(\xi) \times \mathrm{Ker}(\pi_{\alpha, \delta})\right)\\
X_{0}(Np^r)_{/ \mathbb{Q}}&:=X_{1}^{1}(Np^r)_{/ \mathbb{Q}}/\left((\mathbb{Z}/N\mathbb{Z})^{\times}\times\mathbb{Z}_p^{\times}\times\mathbb{Z}_p^{\times}\right).
\end{align*}
\end{definition}

\medskip

\begin{remark}\label{remark 2.2.4}
(1) Since $p \geq 5$, they are smooth projective curves over $\mathbb{Q}$. Therefore, the natural projection maps $$\pi_{1/ \mathbb{Q}}: X_{1}^{1}(Np^r)_{/ \mathbb{Q}}\rightarrow X_{\alpha, \delta, \xi}(Np^r)_{/ \mathbb{Q}} \quad \mathrm{and} \quad \pi_{2/ \mathbb{Q}}:X_{\alpha, \delta, \xi}(Np^r)_{/ \mathbb{Q}}\rightarrow X_{0}(Np^r)_{/ \mathbb{Q}} $$ are finite flat.

\smallskip

(2) We have the following isomorphisms of geometric Galois groups: $$\text{Gal}(X_{1}^{1}(Np^r)/X_{\alpha, \delta, \xi}(Np^r))  \simeq \text{Ker}(\xi) \times H_{\alpha, \delta, \xi, r} \simeq \hat{\Gamma}_{\alpha, \delta, \xi}(Np^r)/\hat{\Gamma}_{1}^{1}(Np^r)$$
 $$\text{Gal}(X_{\alpha, \delta, \xi}(Np^r)/X_{0}(Np^r))   \simeq (\mathbb{Z}/N\mathbb{Z})^{\times} \times \text{Im}(\xi) \times \Gamma / \Gamma^{p^{r-1}} \simeq \hat{\Gamma}_{0}(Np^r) /\hat{\Gamma}_{\alpha, \delta, \xi}(Np^r).$$

\smallskip

(3) The complex points of those curves are given by
\begin{align*}
X_{1}^{1}(Np^r)(\mathbb{C})-\lbrace \text{cusps} \rbrace \simeq \text{GL}_{2}(\mathbb{Q}) \backslash \text{GL}_{2}(\mathbb{A}) /\hat{\Gamma}_{1}^{1}(Np^r) \mathbb{R}^{+}\text{SO}_{2}(\mathbb{R}) \\
X_{\alpha, \delta, \xi}(Np^r)(\mathbb{C})-\lbrace \text{cusps} \rbrace \simeq \text{GL}_{2}(\mathbb{Q}) \backslash \text{GL}_{2}(\mathbb{A}) / \hat{\Gamma}_{\alpha, \delta, \xi}(Np^r)\mathbb{R}^{+}\text{SO}_{2}(\mathbb{R}) \\
X_{0}(Np^r)(\mathbb{C})-\lbrace \text{cusps} \rbrace\simeq\text{GL}_{2}(\mathbb{Q}) \backslash \text{GL}_{2}(\mathbb{A}) /\hat{\Gamma}_{0}(Np^r) \mathbb{R}^{+}\text{SO}_{2}(\mathbb{R}).
\end{align*} Hence $X_{1}^{1}(Np^r), X_{\alpha, \delta, \xi}(Np^r), X_{0}(Np^r) $ are geometrically reduced over $\mathbb{Q}$. Moreover, the projection maps  $$\pi_{1}: X_{1}^{1}(Np^r)\rightarrow X_{\alpha, \delta, \xi}(Np^r) \quad \mathrm{and} \quad \pi_{2}:X_{\alpha, \delta, \xi}(Np^r)\rightarrow X_{0}(Np^r)$$ have constant degree, since the degree is invariant under the flat base change.
\end{remark}

\medskip

\section{Big Ordinary Hecke algebra}

\medskip

Following \cite[Section 4]{hida2015analytic}, we define Hecke algebras associated to the modular curves $X_{\alpha, \delta, \xi}(Np^r)$ in this subsection. For each rational prime $l$, consider $\varpi_l=\begin{pmatrix}
1 & 0 \\
0 & l
\end{pmatrix} \in \text{GL}_{2}(\mathbb{A})$ whose component at prime $q \neq l$ is identity matrix. The group $$\Theta_{r}=\varpi_l^{-1}\widehat{\Gamma}_{\alpha, \delta, \xi}(Np^r)\varpi_l \cap \widehat{\Gamma}_{\alpha, \delta, \xi}(Np^r)$$ give rise to the modular curve $X(\Theta_{r})$ whose complex points are given by 
\begin{align*}
X(\Theta_{r})(\mathbb{C})-\lbrace \text{cusps} \rbrace \simeq \text{GL}_{2}(\mathbb{Q}) \backslash \text{GL}_{2}(\mathbb{A}) /\Theta_{r} \mathbb{R}^{+}\text{SO}_{2}(\mathbb{R}).
\end{align*}

\medskip

We have two projection maps $$\pi_{l, 1}:X(\Theta_{r}) \rightarrow X_{\alpha, \delta, \xi}(Np^r) \quad \mathrm{and} \quad \pi_{l, 2}:X(\Theta_{r}) \rightarrow X_{\alpha, \delta, \xi}(Np^r)$$ defined by $\pi_{l, 1}(z)=z $ and $\pi_{l, 2}(z)=z/l $ for $z \in \mathbb{H}$. Now an embedding $$X(\Theta_{r}) \xrightarrow{\pi_{l, 1} \times \pi_{l, 2}} X_{\alpha, \delta, \xi}(Np^r) \times X_{\alpha, \delta, \xi}(Np^r)$$ defines a modular correspondence. We write this correspondence as $T(l)$ if $l\nmid Np$ and $U(l)$ if $l\mid Np$. Hence we get endomorphisms $$T(l)\quad (l\nmid Np) \quad \mathrm{and} \quad U(l) \quad (l\mid Np)$$ of the Jacobian of $X_{\alpha, \delta, \xi}(Np^r)$ which we denote as $J_{\alpha, \delta, \xi, r}$. (See \cite[Section 4]{hida2015analytic} for the details.) This $J_{\alpha, \delta, \xi, r}$ admits a geometric Galois action of $(\mathbb{Z}/N\mathbb{Z})^{\times} \times (\Gamma \times \Gamma)/\text{Ker}(\pi_{\alpha, \delta})$ acts on $J_{\alpha, \delta, \xi, r}$ via Picard functoriality. 

\medskip

\begin{definition}\label{definition 3.0.1} We define $h_{\alpha, \delta, \xi, r}(\mathbb{Z})$ by the subalgebra of $\text{End}(J_{\alpha, \delta, \xi, r})$ generated by $T(l)$ for $l \nmid Np$, $U(l)$ for $l \mid Np$, and the action of $(\mathbb{Z}/N\mathbb{Z})^{\times} \times (\Gamma \times \Gamma)/\text{Ker}(\pi_{\alpha, \delta})$. We let $h_{\alpha, \delta, \xi, r}(R):=h_{\alpha, \delta, \xi, r}(\mathbb{Z}) \otimes_{\mathbb{Z}}R$ for a commutative ring $R$ and define $\textbf{h}_{\alpha, \delta, \xi, r}:=e(h_{\alpha, \delta, \xi, r}(W))$ where $\displaystyle e:=\lim_{\substack{n \rightarrow  \infty}}U(p)^{n!}$ is the ordinary projector. 
\end{definition}

\medskip

\begin{remark}\label{remark 3.0.2} (1) Since $h_{\alpha, \delta, \xi, r}(W)$ is a finitely generated $W$-module, the ordinary projector $e$ and $\textbf{h}_{\alpha, \delta, \xi, r}$ are well-defined.

(2) For all $r$, $\textbf{h}_{\alpha, \delta, \xi, r}$ is a 
$\Lambda_{\alpha, \delta}$-algebra. For $r \leq s$, we have natural $\Lambda_{\alpha, \delta}$-equivariant surjection $$\textbf{h}_{\alpha, \delta, \xi, s} \twoheadrightarrow \textbf{h}_{\alpha, \delta, \xi, r}.$$
\end{remark}

\medskip

\begin{definition}\label{definition 3.0.3} Define the one variable big ordinary Hecke algebra $\displaystyle \textbf{h}_{\alpha, \delta, \xi}(N):=\lim_{\substack{\longleftarrow \\ r}}\textbf{h}_{\alpha, \delta, \xi, r}$ where the inverse limit is taken with respect to the natural surjections $\textbf{h}_{\alpha, \delta, \xi, s} \twoheadrightarrow \textbf{h}_{\alpha, \delta, \xi, r}$ for $r \leq s$. By Remark \ref{remark 3.0.2}, $\textbf{h}_{\alpha, \delta, \xi}(N)$ is a $\Lambda_{\alpha, \delta}$-algebra.
\end{definition}

\medskip

We quote some facts about structure of the big ordinary Hecke algebra. For the proofs of the following theorems, see \cite[Corollary 4.31]{hida2012p}.

\medskip

\begin{theorem}\label{theorem 3.0.4} (1) Fix a generator $\gamma$ of $ (\Gamma \times \Gamma)/\mathrm{Ker}(\pi_{\alpha, \delta})$. Then we have an isomorphism $$\frac{\textbf{h}_{\alpha, \delta, \xi}(N) }{(\gamma^{p^{r-1}}-1)\textbf{h}_{\alpha, \delta, \xi}(N) } \simeq \textbf{h}_{\alpha, \delta, \xi, r}.$$

(2) $\textbf{h}_{\alpha, \delta, \xi}(N)$ is a $\Lambda_{\alpha, \delta}$-free module of finite rank.
\end{theorem}

\medskip

\section{Limit Mordell-Weil groups and Limit Barsotti-Tate groups}

\medskip

\subsection{Sheaves attached to abelian varieties} \label{sec 4.1}

\begin{definition}\label{definition 4.1.1} Let $K$ be a finite extension of $\mathbb{Q}$ or $\mathbb{Q}_l$ and $A$ be an abelian variety over $K$.

\smallskip

(1) For a finite extension $F/K$, define a finitely generated $W$-module $\displaystyle \widehat{A}(F):=A(F) \otimes_{\mathbb{Z}}W$.

\smallskip

(2) For an algebraic extension $E/K$, define $\displaystyle \widehat{A}(E):=\lim_{\substack{\longrightarrow \\ F}}\widehat{A}(F)$ where direct limit is taken with respect to the natural inclusions.
\end{definition}

\smallskip

Note that if $F$ is a finite extension of $\mathbb{Q}_l$ with $l \neq p$, then $\widehat{A}(F)=A(F)[p^{\infty}]$ is a finite module.\\

We record (without proof) one lemma \cite[Lemma 7.2]{hida2015analytic} regarding the Galois cohomology of $\widehat{A}$.

\begin{lemma}\label{lemma 4.1.2} Let $K$ be a finite extension of $\mathbb{Q}$ or $\mathbb{Q}_l$, $A$ be an abelian variety defined over $K$, and $q$ be a positive integer.

(1) If $K$ is a number field, let $S$ be finite set of places containing infinite places, primes over $p$ and primes of bad reduction of $A$. Then $H^{q}(K^S/K, \widehat{A}) \simeq H^{q}(K^S/K, A)[p^{\infty}] $.

(2) If $K$ is a local field, then $H^{q}(K, \widehat{A}) \simeq H^{q}(K, A)[p^{\infty}] $.
\end{lemma}

Now we define the Selmer group and the Tate-Shafarevich group for $\widehat{A}$.

\begin{definition}\label{definition 4.1.3} Let $K$ be a number field, $A$ be an abelian variety defined over $K$, and $S$ be a finite set of places of $K$ containing infinite places, primes over $p$ and primes of bad reduction of $A$.

(1) $\displaystyle \mathrm{Sel}_{K}(\widehat{A}):=\mathrm{Ker}(H^{1}(K^S/K, \widehat{A}[p^{\infty}]) \rightarrow \prod_{\substack{v \in S}}H^{1}(K_v, \widehat{A}))
$

(2) $\displaystyle \sha_{K}^{1}(\widehat{A}):=\mathrm{Ker}(H^{1}(K^S/K, \widehat{A}) \rightarrow
\prod_{\substack{v \in S}}H^{1}(K_v, \widehat{A}))$

(3) $\displaystyle \sha_{K}^{i}(\widehat{A}[p^{\infty}]):=\mathrm{Ker}(H^{i}(K^S/K, \widehat{A}[p^{\infty}])  \rightarrow
\prod_{\substack{v \in S}}H^{i}(K_v, \widehat{A}[p^{\infty}]))$ for $i=1, 2$
\end{definition}

\begin{remark}\label{remark 4.1.4} (1) By \cite[Corollary $\mathrm{I}.6.6$]{milne2006arithmetic}, this definition is independent of the choice of $S$ as long as $S$ contains infinite places, primes over $p$ and primes of bad reduction of $A$. Moreover, all three modules in the above definition are cofinitely generated $W$-modules.

(2) If we consider the usual classical $p$-adic Selmer group $\mathrm{Sel}_{K}(A):=\mathrm{Ker}(H^{1}(K^S/K, A[p^{\infty}]) \rightarrow \prod_{\substack{v \in S}}H^{1}(K_v, A))$, $\mathrm{Sel}_{K}(\widehat{A})$ is isomorphic to $\mathrm{Sel}_{K}(A)_p$ by Lemma \ref{lemma 4.1.2}. Same assertion holds for the Tate-Shafarevich groups.

(3) Since $A[p^{\infty}] \simeq \widehat{A}[p^{\infty}]$, we get $ \sha_{K}^{1}(\widehat{A}[p^{\infty}]) =\sha_{K}^{1}(A[p^{\infty}])$ and $ \sha_{K}^{2}(\widehat{A}[p^{\infty}])=\sha_{K}^{2}(A[p^{\infty}])$. 
\end{remark}

\medskip

\subsection{Control sequences} \label{sec 4.2}

 Note that the $U(p)$-operator acts on sheaf $\hat{J}_{\alpha, \delta, \xi, r/\mathbb{Q}}$ where $J_{\alpha, \delta, \xi, r/\mathbb{Q}}$ is the Jacobian of modular curve $X_{\alpha, \delta, \xi}(Np^r)_{/\mathbb{Q}}$. By considering the idempotent $\displaystyle e:=\lim_{\substack{n \rightarrow \infty}}U(p)^{n!}$, we define $$\hat{J}^{ord}_{\alpha, \delta, \xi, r/\mathbb{Q}}:=e(\hat{J}_{\alpha, \delta, \xi, r/\mathbb{Q}}).$$ 

\smallskip

\begin{definition}\label{definition 4.2.5} We define the following $\Lambda$-adic Mordell-Weil group and $\Lambda$-adic Barsotti-Tate group as injective limits of sheaves:
(1) $\displaystyle J_{\alpha, \delta, \xi, \infty /\mathbb{Q}}^{ord}:=\lim_{\substack{\longrightarrow \\ n}}\hat{J}_{\alpha, \delta, \xi, n /\mathbb{Q}}^{ord}$ \quad (2) $\displaystyle G_{\alpha, \delta, \xi /\mathbb{Q}}:=\lim_{\substack{\longrightarrow \\ n}}\hat{J}_{\alpha, \delta, \xi, n}^{ord}[p^{\infty}]_{/\mathbb{Q}}$
\end{definition}

\begin{nota}\label{No} Hereafter, we fix the following notations about the Jacobians and the Barsotti-Tate groups. For $\alpha, \delta \in \mathbb{Z}_p$ with $\alpha\mathbb{Z}_p+\delta\mathbb{Z}_p=\mathbb{Z}_p$ and a character $\xi:\mu_{p-1}\times\mu_{p-1} \rightarrow \mu_{p-1}$, we let
\begin{itemize}
\item $J_n=J_{\alpha, \delta, \xi, n}$ and $J'_n=J_{\delta, \alpha, \xi', n}$.
\item $\hat{J}_n^{ord}=\hat{J}_{\alpha, \delta, \xi, n}^{ord}$ and $ \hat{J'}_n^{ord}=\hat{J}_{\delta, \alpha, \xi', n}^{ord}$. 
\item $J_{\infty}^{ord}=J_{\alpha, \delta, \xi, \infty }^{ord}$ and $J_{\infty}^{' ord}=J_{\delta, \alpha, \xi', \infty}^{ord}$
\item $\mathfrak{g}=G_{\alpha, \delta, \xi}$ and $\mathfrak{g}'=G_{\delta, \alpha, \xi'}$.
\end{itemize}
\end{nota}

\begin{remark}\label{remark 4.2.8} (1) Note that natural map $$J_{n /\mathbb{Q}} \rightarrow J_{n+1 /\mathbb{Q}}$$ induced from projection map $X_{\alpha, \delta, \xi}(Np^{n+1})_{/\mathbb{Q}} \twoheadrightarrow X_{\alpha, \delta, \xi}(Np^{n})_{/\mathbb{Q}}$ via Picard functoriality is Hecke-equivariant. Hence we have an induced map $$P_{n, n+1}:\hat{J}_{n /\mathbb{Q}}^{ord} \rightarrow \hat{J}_{n+1 /\mathbb{Q}}^{ord}.$$ In Definition \ref{definition 4.2.5}, direct limit is taken with respect to this $P_{n, n+1}$.

\medskip

(2) If $K$ is an algebraic extension of $\mathbb{Q}$ or $\mathbb{Q}_l$, we give discrete topology on $$\mathfrak{g}(K) \quad \mathrm{and} \quad \hat{J}^{ord}_{r}(K) \otimes \mathbb{Q}_p/\mathbb{Z}_p$$ including $r=\infty$ which makes these two modules continuous $\Lambda$-modules. They are also $\textbf{h}_{\alpha, \delta, \xi}(N)$-modules since the maps $P_{n, n+1}$ are Hecke-equivariant.
\end{remark}

\smallskip

We use the following convention hereafter.

\begin{conv}[Identification]\label{Id} We can find two elements $\gamma_{1}=(a, d)$ and $\gamma_{2}=(d, a)$ of $\Gamma \times \Gamma$ so that they generate $\Gamma \times \Gamma$, $\gamma_{1}=(a, d)$ generates $(\Gamma \times \Gamma) / \mathrm{Ker}(\pi_{\alpha, \delta})$, and $\gamma_{2}=(d, a)$ generates $(\Gamma \times \Gamma) / \mathrm{Ker}(\pi_{\delta, \alpha})$. We identify $\Lambda_{\alpha, \delta}$ with $\Lambda=W[\![T]\!]$ by $(a, d) \leftrightarrow 1+T$ and $\Lambda_{\delta, \alpha}$ with $\Lambda$ by $(d, a) \leftrightarrow 1+T$.
\end{conv}

\medskip

We first state the control results of sheaves $\hat{J}_{r /\mathbb{Q}}^{ord} $ and $\hat{J}_{r}^{ord}[p^{\infty}]_{/\mathbb{Q}}$ with respect to the $p$-power level. For the proof, see \cite[Section 3]{hida2015analytic}.

\begin{prop}\label{prop 4.2.9} For two positive integers $r \leq s$, we have the following isomorphisms of sheaves: 
\begin{align*}
\hat{J}_{r /\mathbb{Q}}^{ord} \simeq \hat{J}_{s}^{ord}[\gamma^{p^{r-1}}-1]_{/\mathbb{Q}} \quad &\mathrm{and} \quad \hat{J}_{r /\mathbb{Q}}^{ord} \simeq J_{\infty}^{ord}[\omega_r]_{/\mathbb{Q}} \\
\hat{J}^{ord}_{r}[p^{\infty}]_{/\mathbb{Q}} \simeq \hat{J}^{ord}_{s}[p^{\infty}][\gamma^{p^{r-1}}-1]_{/\mathbb{Q}} \quad &\mathrm{and} \quad  \hat{J}_{r}^{ord}[p^{\infty}]_{/\mathbb{Q}} \simeq \mathfrak{g}[\omega_r]_{/\mathbb{Q}}
\end{align*}
\end{prop}

\smallskip

\begin{cor}\label{cor 4.2.10} For $K$ either a number field or a finite extension of $\mathbb{Q}_l$,
$\mathfrak{g}(K)$ is a discrete cofinitely generated cotorsion $\Lambda$-module.
\end{cor}

In the Section 6 and 7, we will study the structure of $\mathfrak{g}(K)^{\vee}$ as a finitely generated $\Lambda$-module in detail. From an exact sequence of sheaves $$\displaystyle 0 \rightarrow J_{\infty}^{ord}[\omega_r] \rightarrow J_{\infty}^{ord} \xrightarrow{\omega_{r}}  J_{\infty}^{ord},$$ we get an isomorphism $\displaystyle J_{\infty}^{ord}[\omega_r](K) \simeq J_{\infty}^{ord}(K) [\omega_r]$ where $K$ is either a number field or a finite extension of $\mathbb{Q}_l$, since the global section functor is left exact. Similarly we get $$\mathfrak{g}[\omega_r](K) \simeq \mathfrak{g}(K)[\omega_r].$$

\begin{proof} By Proposition \ref{prop 4.2.9} and the above remark, we have an isomorphism $$\displaystyle \hat{J}^{ord}_{r}[p^{\infty}](K) \simeq \mathfrak{g}(K)[\omega_r].$$ Since $\hat{J}^{ord}_{r}[p^{\infty}](K)$ has finite cardinality due to Mordell-Weil theorem (global case) and Nagell-Lutz theorem (local case), the Nakayama's lemma proves the assertion. 
\end{proof}

\smallskip

\subsection{Cofreeness of Barsotti-Tate groups} \label{sec 4.3}

We state the cofreeness result of Barsotti-Tate groups without proof. (For the proof, see \cite[Section 6]{hida1986galois}.)

\begin{theorem}\label{theorem 4.3.11} (1) $\mathfrak{g}(\overline{\mathbb{Q}})$ is a cofree $\Lambda$-module of finite rank. 

(2) For a complex conjugation $c$, we have $\corank_{\Lambda}(\mathfrak{g}(\overline{\mathbb{Q}})[c-1])=\corank_{\Lambda}(\mathfrak{g}(\overline{\mathbb{Q}})[c+1])$. Hence $\corank_{\Lambda}(\mathfrak{g}(\overline{\mathbb{Q}}))$ is even.
\end{theorem}

\smallskip

\begin{prop}\label{prop 4.3.12} For a prime number $l$, fix an embedding $\overline{\mathbb{Q}} \hookrightarrow \overline{\mathbb{Q}_l}$. Then we have an isomorphism $\mathfrak{g}(\overline{\mathbb{Q}}) \simeq \mathfrak{g}(\overline{\mathbb{Q}_l})$ of $\Lambda$-modules. Hence $\mathfrak{g}(\overline{\mathbb{Q}_l})$ is a cofree $\Lambda$-module.
\end{prop}

As a corollary of Theorem \ref{theorem 4.3.11}, we have following short exact sequences of Barsotti-Tate groups, which will be used frequently in later chapters.

\begin{cor}\label{cor 4.3.13} We have the following exact sequence of sheaves: $$0 \rightarrow \hat{J}^{ord}_{n}[p^{\infty}]_{/\mathbb{Q}} \rightarrow \mathfrak{g}_{/\mathbb{Q}} \xrightarrow{\gamma^{p^{n-1}}-1} \mathfrak{g}_{/\mathbb{Q}} \rightarrow 0.$$
\end{cor}

\begin{proof} By Proposition \ref{prop 4.2.9}, all we need to prove is the exactness at the rightmost term. By Theorem \ref{theorem 4.3.11}, $\mathfrak{g}(\overline{\mathbb{Q}})$ is a cofree $\Lambda$-module, so it is divisible by $\omega_n$. 
\end{proof}

\medskip

\section{Twisted Pairings}

\medskip

We define various pairings between \emph{ordinary} parts of the arithmetic cohomology groups. This section is important for our latter use.

\medskip

For an abelian variety $A$, we let $A^{t}:=\text{Pic}^{0}(A)$ be a dual abelian variety of $A$. If $f:A \rightarrow B$ is an isogeny between two abelian varities $A$ and $B$, we let $f^{t}:B^{t} \rightarrow A^{t}$ be a dual isogeny of $f$.\\

\medskip

We introduce various maps on modular Jacobians and state the relations among those maps. Recall the congruence subgroups 
\begin{align*}
\Gamma_{1, \xi}^{1}(Np^n)=\widehat{\Gamma}_{1, \xi}^{1}(Np^n)\cap \text{SL}_2(\mathbb{Q})\\
\Gamma_{\alpha, \delta, \xi}(Np^n)=\widehat{\Gamma}_{\alpha, \delta, \xi}(Np^n)\cap \text{SL}_2(\mathbb{Q})\\
\Gamma_{0}(Np^n)=\widehat{\Gamma}_{0}(Np^n) \cap \text{SL}_2(\mathbb{Q}).
\end{align*}

\medskip

\subsubsection{Polarization and the Weil involution}

\begin{definition} \label{definition 5.0.1} (1) We let $\lambda_{n}:J_{n} \rightarrow J_{n}^{t}$ be the map descented to $\mathbb{Q}$ from the canonical polarization of the modular Jacobian $J_{n}$.

(2) We let $w_n:J_{n} \rightarrow J'_{n}$ be an isomorphism induced by the Weil involution between $X_{\alpha, \delta, \xi}(Np^n)$ and $X_{\delta, \alpha, \xi'}(Np^n)$. 
\end{definition}

\smallskip

In terms of the double coset operator, we have $$w_n=[\Gamma_{\alpha, \delta, \xi}(Np^n)\backslash \Gamma_{\alpha, \delta, \xi}(Np^n) \begin{pmatrix}
0 & -1 \\
p^n & 0\end{pmatrix} \Gamma_{\delta,\alpha,\xi'}(Np^n)]: J_{n} \rightarrow J'_{n}$$ and the inverse of $w_n$ is given by $$\omega_n^{-1}=[\Gamma_{\delta, \alpha, \xi'}(Np^n)\backslash \Gamma_{\delta, \alpha, \xi'}(Np^n) \begin{pmatrix}
0 & \frac{1}{p^n} \\
-1 & 0\end{pmatrix} \Gamma_{\alpha, \delta, \xi}(Np^n)] : J'_{n} \rightarrow J_{n}.$$\\

\begin{definition} \label{definition 5.0.2}
(1) For a prime $l$, we let $T(l)_{n} $ be a Hecke operator acting on $J_{n}$ and let $T(l)_{n}' $ be a Hecke operator acting on $J_{n}'$. Here we regard $T(l)_{n}=U(l)_{n}$ and $T(l)_{n}'=U(l)_{n}'$ for $l \mid Np$.

(2) We let $T^{*}(l)_{n}$ be a Rosati involution image of $T(l)_{n}$ and $T^{*}(l)_{n}'$ be a Rosati involution image of $T(l)_{n}'$.
\end{definition}

\medskip

We have the following two relations:
\begin{align} \label{e_1} \tag{A}
\lambda_{n}\circ T^{*}(l)_{n}= T(l)_{n}^{t} \circ \lambda_{n}\\
\label{e_2} \tag{B}
 w_n \circ T^*(l)_{n}=T(l)_{n}' \circ w_n
\end{align}

\medskip

\subsubsection{Hecke operators}

\begin{definition}\label{definition 5.0.3} The natural projection map $X_{\alpha, \delta,\xi}(Np^{n+1}) \twoheadrightarrow X_{\alpha, \delta,\xi}(Np^{n})$ induces the following two maps:

\smallskip 

(1) $P_{n, n+1}: J_{n} \rightarrow J_{n+1}$ via contravariant Picard functoriality.

\smallskip

(2) $A_{n+1, n}:J_{n+1} \rightarrow J_{n}$ via Albanese functoriality. Note that this map is well-defined since the natural projection map $X_{\alpha, \delta,\xi}(Np^{n+1}) \twoheadrightarrow X_{\alpha, \delta,\xi}(Np^{n})$ has constant degree. (Remark \ref{remark 2.2.4}-(3).)
\end{definition}

\medskip

We have the following three properties:

\begin{align} \label{e_3} \tag{C}
P_{n, n+1} \circ T(l)_{n}&=T(l)_{n+1} \circ P_{n, n+1}\\
\label{e_4} \tag{D}
A_{n+1, n} \circ T^{*}(l)_{n+1}&=T^{*}(l)_{n} \circ A_{n+1, n}\\
\label{e_5} \tag{E} \lambda_{n} \circ A_{n+1, n}&=P_{n, n+1}^{t} \circ \lambda_{n+1}
\end{align}

\smallskip

\subsubsection{Twisted-Covariant map}

\begin{definition}\label{definition 5.0.4} We define the twisted-covariant map $V_{n+1, n}'=w_n \circ A_{n+1, n} \circ w_{n+1}^{-1}:J'_{n+1} \rightarrow J'_{n}$.
\end{definition}

By (\ref{e_2}) and (\ref{e_4}), this map satisfies :
\begin{align} \label{e_6} \tag{F}
V_{n+1, n}' \circ T(l)_{n+1}'=T(l)_{n}' \circ V_{n+1, n}'
\end{align}

\smallskip

By \cite[Section 4]{hida2013limit} and \cite[Section 4]{hida2013lambda}, we have:

\begin{align} \label{e_7} \tag{G} V_{n+1, n}'=1+\gamma+\gamma^{2}+\cdots+\gamma^{p-1}:J'_{n+1} \rightarrow J'_{n}
\end{align}

where $\gamma$ is a generator of geometric Galois group $\text{Gal}(X_{\delta, \alpha, \xi'}(Np^{n+1})/X_{\delta, \alpha, \xi'}(Np^n))$.

\smallskip

\subsubsection{Diamond operators}

Note that $\Gamma_{0}(Np^n)$ acts on $X_{\alpha, \delta, \xi}(Np^n)$ by the geometric Galois action. Action of $h \in \Gamma_{0}(Np^n)$ induces an automorphism of $J_n$ over $\mathbb{Q}$ via Picard functoriality which can be written in terms of the double coset operator as $$[\Gamma_{\alpha, \delta, \xi}(Np^n) \backslash \Gamma_{\alpha, \delta, \xi}(Np^n) h \Gamma_{\alpha, \delta, \xi}(Np^n)]:J_{n} \xrightarrow{\sim} J_{n}.$$ Similarly, same $h \in \Gamma_{0}(Np^n)$ induces another isomorphism $$[\Gamma_{\alpha, \delta, \xi}(Np^n) \backslash \Gamma_{\alpha, \delta, \xi}(Np^n) h^{-1} \Gamma_{\alpha, \delta, \xi}(Np^n)]:J_{n} \xrightarrow{\sim} J_{n}$$ defined over $\mathbb{Q}$ via Albanese functoriality.\\

By the same reason with (\ref{e_5}), we have the following identity for $g \in \Gamma_{0}(Np^n)$:
\begin{align} \label{e_8} \tag{H}
\lambda_{n} \circ [\Gamma_{\alpha, \delta, \xi}(Np^n) \backslash \Gamma_{\alpha, \delta, \xi}(Np^n) g^{-1} \Gamma_{\alpha, \delta, \xi}(Np^n)]
=[\Gamma_{\alpha, \delta, \xi}(Np^n) \backslash \Gamma_{\alpha, \delta, \xi}(Np^n) g \Gamma_{\alpha, \delta, \xi}(Np^n)]^{t} \circ \lambda_{n}
\end{align}

For $g \in \Gamma_{0}(Np^n)$, we let $\tilde{g}=\begin{pmatrix}
0 & -1 \\
p^n & 0\end{pmatrix} 
g^{-1} 
\begin{pmatrix}
0 & \frac{1}{p^n} \\
-1 & 0
\end{pmatrix}$ for the moment. Note that $\tilde{g} \in \Gamma_{0}(Np^n)$.

We then have the following identity which can be checked easily: (For the moment, we let $\Xi=\Gamma_{\alpha, \delta, \xi}(Np^n)$ and $\Xi'=\Gamma_{\delta, \alpha, \xi'}(Np^n)$)
\begin{align} \label{e_9} \tag{I}
[\Xi \backslash \Xi g^{-1} \Xi] \circ [\Xi' \backslash \Xi'\begin{pmatrix}
0 & \frac{1}{p^n} \\
-1 & 0
\end{pmatrix} \Xi]
=[\Xi' \backslash \Xi' 
\begin{pmatrix}
0 & \frac{1}{p^n} \\
-1 & 0 \end{pmatrix} \Xi] \circ 
[\Xi' \backslash \Xi'
\tilde{g}
\Xi']
\end{align}\\

We can rewrite the above identity as:
\begin{align} \label{e_10} \tag{J}
[\Gamma_{\alpha, \delta, \xi}(Np^n) \backslash \Gamma_{\alpha, \delta, \xi}(Np^n) g^{-1} \Gamma_{\alpha, \delta, \xi}(Np^n)] \circ w_n^{-1}
=w_n^{-1} \circ 
[\Gamma_{\delta, \alpha, \xi'}(Np^n) \backslash \Gamma_{\delta, \alpha, \xi'}(Np^n)
\tilde{g}
\Gamma_{\delta, \alpha, \xi'}(Np^n)]
\end{align}

\medskip

\subsection{Twisted Weil pairing} \label{sec 5.1}

\medskip

Let $W^{n}_{m}:J_{n}[p^m] \times J_{n}^{t}[p^m] \rightarrow \mu_{p^m}$ be the Weil pairing, which is perfect Galois-equivariant bilinear. We define a twisted Weil pairing $$H_{m}^{n}:J_{n}[p^m] \times J'_ {n}[p^m] \rightarrow \mu_{p^m}$$ by $$H_{m}^{n}(x, y)=W_{m}^{n}(x, \lambda_{n} \circ w_{n}^{-1}(y))$$ where $x, y$ are elements of $J_{n}[p^m], J'_ {n}[p^m]$, respectively. We record the basic properties of this twisted pairing.\\

\begin{prop}\label{prop 5.1.5} (1) $H_{m}^{n}$ is a perfect, Galois-equivariant bilinear pairing. 

\medskip

(2) $(\mathrm{Hecke-equivariance})$ $H_{m}^{n}\left(T(l)_{n}(x), y\right)=H_{m}^{n}\left(x, T(l)_{n}'(y)\right)$ where $x, y$ are elements of $J_{n}[p^m], J'_ {n}[p^m]$, respectively. Hence, $H_{m}^{n}$ induces a perfect pairing  $$H_{m}^{n}:\hat{J}_{n}^{ord}[p^m] \times \hat{J'}_ {n}^{ord}[p^m] \rightarrow \mu_{p^m}.$$

\medskip

(3) For $x \in \hat{J}^{ord}_{n}[p^m]$ and $y \in \hat{J'}^{ord}_{n}[p^{m+1}]$, $H_{m+1}^{n}(x, y)=H_{m}^{n}(x, py) $. Hence there is a perfect Galois-equivariant bilinear pairing $\hat{J}^{ord}_{n}[p^{\infty}] \times T_{p}J^{' ord}_{n}\rightarrow \mu_{p^{\infty}}$.

\medskip
\medskip

(4) For $x \in \hat{J}^{ord}_{n}[p^m]$ and $y \in \hat{J'}^{ord}_{n+1}[p^{m}]$, $H_{m}^{n+1}( P_{n, n+1}(x), y)=H_{m}^{n}(x, V_{n+1, n}'(y))$ holds.

\medskip
\medskip

(5) Take $g \in \Gamma_{0}(Np^n)$ and let $\tilde{g}=\begin{pmatrix}
0 & -1 \\
p^n & 0\end{pmatrix} 
g^{-1} 
\begin{pmatrix}
0 & \frac{1}{p^n} \\
-1 & 0
\end{pmatrix}$. If $x, y$ are elements of $\hat{J}^{ord}_{n}[p^m], \hat{J'}^{ord}_{n}[p^m]$, respectively, then we have
$$H_{m}^{n} \left( [\Gamma_{\alpha, \delta, \xi}(Np^n)\backslash\Gamma_{\alpha, \delta, \xi}(Np^n) g \Gamma_{\alpha, \delta, \xi}(Np^n)] x, y \right)=H_{m}^{n} \left( x,[\Gamma_{\delta, \alpha, \xi'}(Np^n) \backslash \Gamma_{\delta, \alpha, \xi'}(Np^n) \tilde{g} \Gamma_{\delta, \alpha, \xi'}(Np^n)]y \right).$$
\end{prop}

\medskip

\begin{proof} Since (5) directly follows from \eqref{e_10},  we prove (2) and (4) only.

For (2), 
\begin{align*} H_{m}^{n}(T(l)_{n}(x), y)
&=W_{m}^{n}(T(l)_{n}(x), \lambda_{n} \circ w_{n}^{-1}(y))\tag{Definition}\\
&=W_{m}^{n}(x, (T(l)_{n})^{t} \circ \lambda_{n} \circ w_{n}^{-1}(y))\tag{Property of Weil pairing}\\
&=W_{m}^{n}(x, \lambda_{n} \circ T^{*}(l)_{n} \circ w_{n}^{-1}(y))\tag{By \eqref{e_2}}\\
&=W_{m}^{n}(x, \lambda_{n} \circ w_{n}^{-1} \circ T(l)_{n}'(y))\tag{By \eqref{e_3}}\\
&=H_{m}^{n}(x, T(l)_{n}'(y)) \tag{Definition}
\end{align*}

For (4), 
\begin{align*} H_{m}^{n+1}( P_{n, n+1}(x), y)&=W_{m}^{n+1}( P_{n, n+1}(x), \lambda_{n+1} \circ w_{n+1}^{-1}(y))\tag{Definition}\\
&=W_{m}^{n}(   x, P_{n, n+1}^{t}\circ \lambda_{n+1} \circ w_{n+1}^{-1}(y))\tag{Property of Weil pairing}\\
&=W_{m}^{n}(   x, \lambda_{n} \circ A_{n+1, n} \circ w_{n+1}^{-1}(y))\tag{By \eqref{e_6}} \\
&=W_{m}^{n}(   x, \lambda_{n} \circ w_{n}^{-1}  \circ V'_{n+1, n}(y)) \\
&=H_{m}^{n}(x, V_{n+1, n}'(y))\tag{Definition}
\end{align*}
\end{proof}

For $g=\begin{pmatrix}
a & b \\
cp^n & d\end{pmatrix} \in \Gamma_{0}(Np^n)$, $\tilde{g}=\begin{pmatrix}
0 & -1 \\
p^n & 0\end{pmatrix} 
g^{-1} 
\begin{pmatrix}
0 & \frac{1}{p^n} \\
-1 & 0
\end{pmatrix}=\frac{1}{ad-bcp^n}\begin{pmatrix}
a & c \\
bp^n & d\end{pmatrix}$ and this last matrix is congruent to $\begin{pmatrix}
\frac{1}{d} & \frac{c}{ad} \\
0 & \frac{1}{a}\end{pmatrix}$ modulo $p^n$. Hence we have the following corollary about the $\Lambda$-equivariance of our twisted pairing.\\

\begin{cor}\label{cor 5.1.6} The twisted Weil pairing $H_{m}^{n}:\hat{J}_{n}^{ord}[p^m] \times \hat{J'}_ {n}^{ord}[p^m] \rightarrow \mu_{p^m}$ satisfies $$H_{m}^{n}(f \cdot x, y)=H_{m}^{n}(x, \iota(f)\cdot y)$$ where $x, y$ are elements of $\hat{J}^{ord}_{n}[p^m], \hat{J'}^{ord}_ {n}[p^m]$, respectively, and $f \in W[\![T]\!]$. 

In particular, we have an isomorphism $\hat{J}_{n}^{ord}[p^m] \simeq \hat{J'}_ {n}^{ord}[p^m]^{\vee}$ of $\Lambda$-modules.
\end{cor}

\smallskip

\subsection{Local Tate duality} \label{sec 5.2}

We have the following local Tate duality induced from the twisted Weil pairing.

\smallskip

\begin{prop}[Local Tate duality]\label{prop 5.2.7} Let $L$ be a finite extension of $\mathbb{Q}_l$ for some prime $l$.

(1) For $i=0, 1, 2$, pairings $H^{i}(L, \hat{J}^{ord}_{n}[p^m] ) \times H^{2-i}(L,\hat{J'}^{ord}_{n}[p^{m}] )^{\iota} \rightarrow Q(W)/W$ are perfect $\Lambda$-equivariant bilinear pairing between two finite $p$-abelian groups.

\smallskip

(2) The images of Kummer maps $\frac{\hat{J}^{ord}_{n}(L)}{p^{m}\hat{J}^{ord}_{n}(K)}$ and $\frac{\hat{J'}^{ord}_{n}(L)}{p^{m}\hat{J'}^{ord}_{n}(L)}$ into cohomology groups are orthogonal complement to each other under the pairing $H^{1}(L, \hat{J}^{ord}_{n}[p^m] ) \times H^{1}(L,\hat{J'}^{ord}_{n}[p^{m}]) \rightarrow Q(W)/W$ of (1).
\end{prop}

\medskip

By this local Tate duality, we can construct the following perfect pairing between two $\Lambda$-modules.

\smallskip

\begin{cor}\label{cor 5.2.8} Let $L$ be a finite extension of $\mathbb{Q}_l$ for some prime $l$.

(1) There is a perfect bilinear $\Lambda$-equivariant pairing $\hat{J}^{ord}_{n}[p^{\infty}](L) \times H^{1}(L,\hat{J'}^{ord}_{n}[p^{\infty}] )^{\iota}_{/ div} \rightarrow Q(W)/W$ between finite $p$-abelian groups. 

\medskip

(2) The above pairing induces a perfect $\Lambda$-equivariant pairing $$\displaystyle \mathfrak{g}(L) \times \lim_{\substack{\longleftarrow \\n}}H^{1}(L,\hat{J'}^{ord}_{n}[p^{\infty}] )^{\iota}_{/ div} \rightarrow Q(W)/W$$ where the inverse limit is taken with respect to $V_{n+1, n}'$. 
\end{cor}

\smallskip

\begin{proof} Since $L$ is local, we have $H^{2}(L,\hat{J'}^{ord}_{n}[p^{\infty}] )=0$. Hence we get $$\frac{H^{1}(L,\hat{J'}^{ord}_{n}[p^{\infty}] )}{p^{t}H^{1}(L,\hat{J'}^{ord}_{n}[p^{\infty}] )} \simeq H^{2}(L,\hat{J'}^{ord}_{n}[p^{t}] ) \simeq \hat{J}^{ord}_{n}[p^{t}](L) ^{\vee}$$ for all $t$. For the last isomorphism, we used Proposition \ref{prop 5.2.7}-(1). Passing to the projective limit with respect to $t$ gives the desired assertion.
\end{proof}

\smallskip

\subsection{Poitou-Tate duality} \label{sec 5.3}

We have the (global) Poitou-Tate duality (See \cite[Theorem I.4.10]{milne2006arithmetic} for the reference) induced from the twisted Weil pairing. 

\begin{prop}[Poitou-Tate duality]\label{prop 5.3.9} There is a perfect $\Lambda$-equivariant pairing of finite $p$-abelian groups $$\sha ^{1}_{K}(\hat{J}_{n}^{ord}[p^m]) \times \sha^{2}_{K}(\hat{J'}_{n}^{ord}[p^m])^{\iota} \rightarrow Q(W)/W.$$
\end{prop}

Later in the subsection \ref{sec 7.2}, we will compare the $\Lambda$-corank of $\sha^{1}_{K}(\mathfrak{g}) $ and $\sha^{2}_{K}(\mathfrak{g}') $ by using this pairing. (See Proposition \ref{prop 7.2.8})

\subsection{Twisting pairing of Flach and Cassels-Tate} \label{sec 5.4}

\subsubsection{Description of the original pairings}
We briefly recall properties of pairing of Flach and Cassels-Tate. 

\medskip

Let $K$ be a number field, $A$ be an abelian variety defined over $K$, and let $S$ be a finite set of places of $K$ containing infinite places, places over $p$ and the places of bad reduction of $A$. We recall the Selmer group and the Tate-Shafarevich group associated to $\widehat{A}$ (Definition \ref{definition 4.1.3}):
\begin{align*}
\displaystyle \mathrm{Sel}_{K}(\widehat{A}):=\mathrm{Ker}(H^{1}(K^S/K, \widehat{A}[p^{\infty}]) \rightarrow \prod_{\substack{v \in S}}H^{1}(K_v, \widehat{A}))
\\
\displaystyle \sha_{K}^{1}(\widehat{A}):=\mathrm{Ker}(H^{1}(K^S/K, \widehat{A}) \rightarrow
\prod_{\substack{v \in S}}H^{1}(K_v, \widehat{A}))
\end{align*}

As we remarked before, these definitions are independent of the choice of $S$ as long as $S$ is finite and $S$ contains infinite places, places over $p$ and the places of bad reduction of $A$.

\begin{remark}\label{remark 5.4.10} If $A$ is a modular Jacobian $J_{n}$, then $\text{Sel}_{K}(\widehat{J_{n}})$ and $\sha^{1}_{K}(\widehat{J_{n}})$ are $\textbf{h}_{\alpha, \delta, \xi}$-modules and we have $\text{Sel}_{K}(\widehat{J_{n}})^{ord}=\text{Sel}_{K}(\hat{J}^{ord}_{n})$ and $\sha^{1}_{K}(\widehat{J_{n}})^{ord}=\sha^{1}_{K}(\hat{J}^{ord}_{n})$.
\end{remark}

\smallskip

Flach (\cite{flach1990generalisation}) constructed a bilinear pairing $F_{A}:\mathrm{Sel}_{K}(\widehat{A}) \times \mathrm{Sel}_{K}(\widehat{A^t}) \rightarrow Q(W)/W$ whose left kernel (resp. right kernel) is the maximal $p$-divisible subgroup of $\mathrm{Sel}_{K}(\widehat{A})$ (resp. $\mathrm{Sel}_{K}(\widehat{A^t})$). Moreover, for an isogeny $f:A \rightarrow B$ defined over $K$, we have $$F_{A}(x, f^{t}(y))=F_{B}(f(x), y)$$ for $x \in \mathrm{Sel}_{K}(\widehat{A})$ and $y \in \mathrm{Sel}_{K}(\widehat{B^t})$. Note that the Cassels-Tate pairing $C_{A}:\sha^{1}_{K}(\widehat{A}) \times \sha^{1}_{K}(\widehat{A^t})\rightarrow Q(W)/W$ also satisfies the same functorial property. (See \cite[Theorem I.6.26]{milne2006arithmetic})

\medskip

\subsubsection{Twisted pairing}

For the modular Jacobian $J_{n}$, we define the twisted Flach's pairing and the twisted Cassels-Tate pairing by
$$
S^{n}(x, y)=F_{J_{n}}(x, \lambda_{n} \circ w_{n}^{-1}(y) )
$$
for $x \in \mathrm{Sel}_{K}(\widehat{J_{n}})$ and $y \in \mathrm{Sel}_{K}(\widehat{J'_{n}})$ and
$$
CT^{n}(z, v)=C_{J_{n}}(z, \lambda_{n} \circ w_{n}^{-1}(v) )
$$

for $z \in \sha^{1}_{K}(\widehat{J_{n}})$ and $v \in \sha^{1}_{K}(\widehat{J'_{n}})$. We record properties of the pairing $S^n$ below. (Note that the twisted Cassels-Tate pairing $CT^n$ also satisfies the same properties.)

\medskip

\begin{prop}\label{prop 5.4.11} 
(1) $(\mathrm{Hecke-equivariance})$ $S^{n}(T(l)_{n}(x), y)=S^{n}(x, T(l)_{n}'(y))$ where $x, y$ are elements of $\mathrm{Sel}_{K}(\widehat{J_{n}}), \mathrm{Sel}_{K}(\widehat{J'_{n}})$, respectively. Hence, $S^{n}$ induces a bilinear pairing $$S^{n}:\mathrm{Sel}_{K}(\hat{J}^{ord}_{n}) \times \mathrm{Sel}_{K}(\hat{J'}^{ord}_{n}) \rightarrow Q(W)/W.$$

\medskip

(2) $S^{n}: \mathrm{Sel}_{K}(\hat{J}^{ord}_{n})_{/ div} \times \mathrm{Sel}_{K}(\hat{J'}^{ord}_{n})_{/ div} \rightarrow Q(W)/W$ is a perfect bilinear pairing.

\medskip

(3) For $x \in \mathrm{Sel}_{K}(\hat{J}^{ord}_{n})$ and $y \in \mathrm{Sel}_{K}(\hat{J'}^{ord}_{n+1})$, we have $$S^{n+1}\left( P_{n, n+1}  (x), y\right)=S^{n}\left(x, V_{n+1, n}'(y)\right).$$ 

\medskip

(4) The twisted Flach's pairing $S^{n}$ satisfies $$S^{n}(f \cdot x, y)=S^{n}(x, \iota(f)\cdot y)$$ where $x, y$ are elements of $\mathrm{Sel}_{K}(\hat{J}^{ord}_{n}), \mathrm{Sel}_{K}(\hat{J'}^{ord}_{n})$, respectively and $f \in W[\![T]\!]$.

In particular, we have an isomorphism $\mathrm{Sel}_{K}(\hat{J}^{ord}_{n})_{/ div} \simeq \left(\mathrm{Sel}_{K}(\hat{J'}^{ord}_{n})_{/ div}\right)^{\vee}$ of $\Lambda$-modules.
\end{prop}

\medskip

Since the proof of Proposition \ref{prop 5.4.11} is almost identical with those of Proposition \ref{prop 5.1.5} and Corollary \ref{cor 5.1.6}, we omit the proof. 

\smallskip

\section{$\Lambda$-adic local cohomology groups}

\medskip

In this section, we let $L$ be a finite extension of $\mathbb{Q}_l$. (Here $l$ can be same with $p$ unless we mention that $l \neq p$.) We keep Notation \ref{No} and Convention \ref{Id}.

\smallskip

\subsection{Basic facts} \label{sec 6.1}

Note that $H^{2}(L, A)=0$ for any abelian variety defined over $L$ due to \cite[Theorem I.3.2]{milne2006arithmetic}. By this fact, we get $H^{2}(L, A[p^{\infty}])=0$.

\begin{lemma}\label{lemma 6.1.1} We have the following assertions:

(1) $\mathfrak{g}(L)^{\vee}$ is a finitely generated torsion $\Lambda$-module whose characteristic ideal is prime to $\omega_{n}$ for all $n$. In particular, $\frac{\mathfrak{g}(L)}{\omega_{n}\mathfrak{g}(L)}$ has finite bounded order independent of $n$. We have $\displaystyle \lim_{\substack{\longrightarrow \\n}}\frac{\mathfrak{g}(L)}{\omega_{n}\mathfrak{g}(L)}=0$. Moreover $\displaystyle \lim_{\substack{\longleftarrow \\n}}\frac{\mathfrak{g}(L)}{\omega_{n}\mathfrak{g}(L)}$ is isomorphic to Pontryagin dual of the maximal finite submodule of $\mathfrak{g}(L)^{\vee}$ as $\Lambda$-modules.

(2) $H^{2}(L, \mathfrak{g})=0$.

(3) $H^{1}(L, \mathfrak{g})^{\vee}[\omega_{n}]=0$. Hence $H^{1}(L, \mathfrak{g})^{\vee}$ has no non-trivial finite $\Lambda$-submodules and the characteristic ideal of $(H^{1}(L, \mathfrak{g})^{\vee})_{\Lambda-\mathrm{tor}}$ is coprime to $\omega_{n}$ for all $n$.
\end{lemma}

\begin{proof} By the control sequence (Proposition \ref{prop 4.2.9}), we have $\hat{J}_{n}^{ord}[p^{\infty}](L)=\mathfrak{g}(L)[\omega_{n}]$ and this group is finite since $L$ is a finite extension of $\mathbb{Q}_l$ (Nagell-Lutz). So the first assertion follows from Lemma \ref{lemma A.0.1}. For (2), $H^{2}(L, \hat{J}_{n}^{ord}[p^{\infty}])=0$ as we remarked earlier, so we have $\displaystyle H^{2}(L, \mathfrak{g})=\lim_{\substack{\longrightarrow \\ n}}H^{2}(L, \hat{J}_{n}^{ord}[p^{\infty}])=0$.

For (3), $\frac{H^{1}(L, \mathfrak{g})}{\omega_{n}H^{1}(L, \mathfrak{g})}$ injects into $H^{2}(L, \hat{J}_{n}^{ord}[p^{\infty}])=0$  so $H^{1}(L, \mathfrak{g})^{\vee}[\omega_{n}]=0$. Now Lemma \ref{lemma A.0.1}-(3) proves the remaining assertions of (3).
\end{proof}

Next we calculate the Euler characteristic of the cohomology groups above. 

\medskip

\begin{lemma}[Local Euler Characteristic Formula]\label{lemma 6.1.2}

\begin{numcases}{(1) \quad  \corank_{W}H^{1}(L, \hat{J}_{n}^{ord}[p^{\infty}])=}
p^{n-1}\cdot[L:\mathbb{Q}_l]\cdot \corank_{\Lambda}\mathfrak{g}(\overline{L}) & $l=p$  \nonumber\\   
   0 & $l\neq p$  \nonumber
   \end{numcases}

\begin{numcases}
{(2) \quad   \corank_{\Lambda}H^{1}(L, \mathfrak{g})=}
[L:\mathbb{Q}_l]\cdot \corank_{\Lambda}\mathfrak{g}(\overline{L}) & $l=p$  \nonumber\\
   0 & $l\neq p$  \nonumber
\end{numcases}
\end{lemma}

\begin{proof} By Proposition \ref{prop 4.2.9}, we have an isomorphism $\hat{J}_{n}^{ord}[p^{\infty}](\overline{L})\simeq \mathfrak{g}(\overline{L})[\omega_{n}]$. Hence $\hat{J}_{n}^{ord}[p^{\infty}](\bar{L})$ is a cofree $W$-module of corank $p^{n-1}\cdot \corank_{\Lambda}\mathfrak{g}(\overline{L})$ by Theorem \ref{theorem 4.3.11}-(1). Now the usual Euler characteristic formula shows (1). (Note that $\hat{J}_{n}^{ord}[p^{\infty}](L)$ is finite.)

For (2), we have the following short exact sequence by Corollary \ref{cor 4.3.13}:
$$0 \rightarrow \frac{\mathfrak{g}(L)}{\omega_{n}\mathfrak{g}(L)} \rightarrow H^{1}(L, \hat{J}_{n}^{ord}[p^{\infty}]) \rightarrow H^{1}(L, \mathfrak{g})[\omega_{n}] \rightarrow 0.$$ Since $\frac{\mathfrak{g}(L)}{\omega_{n}\mathfrak{g}(L)}$ is a finite group by Lemma \ref{lemma 6.1.1}-(1), comparing $W$-coranks proves the second assertion. 
\end{proof}

\medskip

\subsection{Functors $\mathfrak{F}$ and $\mathfrak{G}$} \label{sec 6.2}

We record the definitions and properties of the functors $\mathfrak{F}$ and $\mathfrak{G}$ defined in \cite{Lee2018}.

\begin{definition}\label{definition 6.2.3} For a finitely generated $\Lambda$-module $X$, we define \begin{align*}
\mathfrak{F}(X)&:=\left(\lim_{\substack{\longrightarrow \\n}}\frac{X}{\omega_nX}[p^{\infty}]\right)^{\vee}\\
\mathfrak{G}(X)&:=\lim_{\substack{\longleftarrow \\n}}\left(\frac{X}{\omega_nX}[p^{\infty}] \right).
\end{align*}
\end{definition}

Here the direct limit is taken with respect to the norm maps $\frac{X}{\omega_{n}X} \xrightarrow{\times\frac{\omega_{n+1}}{\omega_{n}}} \frac{X}{\omega_{n+1}X}$, and the inverse limit is taken with respect to the natural projections. We have the following proposition about the explicit description of two functors $\mathfrak{F}$ and $\mathfrak{G}$. For the proof, see \cite[Proposition A.1.6, Proposition A.2.12]{Lee2018}.

\medskip

\begin{prop}\label{prop 6.2.4}
Let $X$ be a finitely generated $\Lambda$-module with $$\displaystyle E(X)  \simeq \Lambda^{r} 
\oplus \left( \bigoplus_{i=1}^{d} \frac{\Lambda}{g_{i}^{l_{i}}} \right) 
\oplus \left( \bigoplus_{\substack{m=1 \\ e_{1} \cdots e_{f} \geq 2}}^{f} \frac{\Lambda}{\omega_{a_{m}+1, a_{m}}^{e_{m}}} \right) 
\oplus \left( \bigoplus_{n=1}^{t} \frac{\Lambda}{\omega_{b_{n}+1, b_{n}}} \right) $$ where $r \geq 0$, $g_{1}, \cdots g_{d}$ are prime elements of $\Lambda$ which are coprime to $\omega_{n}$ for all $n$, $d \geq 0$, $l_{1}, \cdots, l_{d} \geq 1$, $f \geq 0$, $e_{1}, \cdots, e_{f} \geq 2$ and $t \geq 0$.

\medskip

(1) We have an injection $$\displaystyle \left(\bigoplus_{i=1}^{d} \frac{\Lambda}{\iota(g_{i})^{l_{i}}}\right)\oplus \left(\bigoplus_{\substack{m=1 \\ e_{1} \cdots e_{f} \geq 2}}^{f} \frac{\Lambda}{\iota(\omega_{a_{m}+1, a_{m}})^{e_{m}-1}} \right) \hookrightarrow \mathfrak{F}(X)$$ with finite cokernel. In particular, $\mathfrak{F}(X)$ is a finitely generated $\Lambda$-torsion module.

\medskip

(2) We have a pseudo-isomorphism $$\displaystyle \mathfrak{G}(X) \xrightarrow{\mathfrak{G}(\phi)} \left(\bigoplus_{i=1}^{d} \frac{\Lambda}{g_{i}^{l_{i}}}\right)\oplus \left(\bigoplus_{\substack{m=1 \\ e_{1} \cdots e_{f} \geq 2}}^{f} \frac{\Lambda}{\omega_{a_{m}+1, a_{m}}^{e_{m}-1}} \right).$$ In particular, $\mathfrak{G}(X)$ is a finitely generated $\Lambda$-torsion module.

\medskip

(3) Two $\Lambda$-torsion modules $\mathfrak{F}(X)^{\iota}$ and $\mathfrak{G}(X)$ are pseudo-isomorphic.
\end{prop}

\medskip

We mention one corollary which will be used in Corollary \ref{cor 6.3.7}.

\medskip

\begin{cor}\label{cor 6.2.5} Let $X$ be a finitely generated $\Lambda$-module and denote $X_{\Lambda-\mathrm{tor}}$ as the maximal $\Lambda$-torsion submodule of $X$. If the characteristic ideal of $X_{\Lambda-\mathrm{tor}}$ is coprime to $\omega_n$ for all $n$, then there is a pseudo-isomorphism $\phi:X \rightarrow \Lambda^{\rank_{\Lambda} X} \oplus \mathfrak{F}(X)^{\iota}$. If we assume additionally that $X$ does not have any non-trivial finite $\Lambda$-submodules, then $\phi$ is an injection.
\end{cor}

\smallskip

\medskip

\subsection{Application of the local Tate duality} \label{sec 6.3}

\medskip

By applying the functor $\mathfrak{F}$ to $H^{1}(L, \mathfrak{g})^{\vee}$, we can compare the $\Lambda$-module structure of the local cohomology groups of the Barsotti-Tate groups of two different towers. This can be regarded as an analogue of \cite[Proposition 1, 2]{greenberg1989iwasawa}. 

\medskip

\begin{prop}\label{prop 6.3.6} There is a surjection $\phi: \mathfrak{g}'(L)^{\vee} \twoheadrightarrow \mathfrak{F}\left(H^{1}(L, \mathfrak{g})^{\vee}\right)$ with finite kernel. In particular, two $\Lambda$-torsion modules $\mathfrak{g} '(L)^{\vee}$ and $\mathfrak{F}\left(H^{1}(L, \mathfrak{g})^{\vee}\right)$ are pseudo-isomorphic.
\end{prop} 

\smallskip

\begin{proof} By Corollary \ref{cor 4.3.13}, we have a short exact sequence $$0 \rightarrow \frac{\mathfrak{g}(L)}{\omega_{n}\mathfrak{g}(L)} \rightarrow H^{1}(L, \hat{J}_{n}^{ord}[p^{\infty}]) \rightarrow H^{1}(L, \mathfrak{g})[\omega_{n}] \rightarrow 0.$$ Since $\frac{\mathfrak{g}(L)}{\omega_{n}\mathfrak{g}(L)}$ is finite, we get $\frac{\mathfrak{g}(L)}{\omega_{n}\mathfrak{g}(L)} \rightarrow H^{1}(L, \hat{J}_{n}^{ord}[p^{\infty}])_{/ div} \rightarrow \left(H^{1}(L, \mathfrak{g})[\omega_{n}]\right)_{/ div} \rightarrow 0$ by Lemma \ref{lemma A.0.2}-(3). Note that all terms in this sequence are finite. Taking the projective limit gives the following exact sequence: 

$$\displaystyle \lim_{\substack{\longleftarrow \\n}}\frac{\mathfrak{g}(L)}{\omega_{n}\mathfrak{g}(L)} \rightarrow \lim_{\substack{\longleftarrow \\n}}H^{1}(L, \hat{J}_{n}^{ord}[p^{\infty}])_{/ div} \rightarrow \lim_{\substack{\longleftarrow \\n}}\left(H^{1}(L, \mathfrak{g})[\omega_{n}]\right)_{/ div} \rightarrow 0.$$ 

\smallskip

By Lemma \ref{lemma 6.1.1}-(1), the first term is isomorphic to Pontryagin dual of the maximal finite submodule of $\mathfrak{g}(L)^{\vee}$. For the second term, we have an isomorphism  $\displaystyle \lim_{\substack{\longleftarrow \\n}}H^{1}(L, \hat{J}_{n}^{ord}[p^{\infty}])_{/ div} \simeq \mathfrak{g}'(L)^{\vee}$ by Corollary \ref{cor 5.2.8}-(2). The last term is $\mathfrak{F}\left(H^{1}(L, \mathfrak{g})^{\vee}\right)$ by definition.
\end{proof}

\medskip

\begin{cor}\label{cor 6.3.7} (1) If $l \neq p$, then there is an injective map $H^{1}(L, \mathfrak{g})^{\vee}\hookrightarrow \left(\mathfrak{g}'(L)^{\vee}\right)^{\iota}$ with finite cokernel.

\medskip

(2) If $l=p$, then there is an injective map $H^{1}(L, \mathfrak{g})^{\vee}\hookrightarrow {\Lambda}^{[L:\mathbb{Q}_p]\cdot \corank_{\Lambda}\mathfrak{g}(\overline{L}) } \oplus \left(\mathfrak{g}'(L)^{\vee}\right)^{\iota}$ with finite cokernel.
\end{cor}

\smallskip

\begin{proof} (1) and (2) follows from Lemma \ref{lemma 6.1.1}-(3), Proposition \ref{prop 6.3.6} and Corollary \ref{cor 6.2.5}.
\end{proof}

\medskip

\begin{cor}\label{cor 6.3.8} (1) If $l \neq p$, $\mathfrak{g}'(L)$ is finite if and only if $H^{1}(L, \mathfrak{g})^{\vee}=0$.

\medskip

(2) If $l=p$, $\mathfrak{g}'(L)$ is finite if and only if $H^{1}(L, \mathfrak{g})^{\vee}$ is $\Lambda$-torsion-free module if and only if $H^{1}(L, \mathfrak{g})^{\vee}$ is a $\Lambda$-submodule of ${\Lambda}^{[L:\mathbb{Q}_p]\cdot \corank_{\Lambda}\mathfrak{g}(\overline{L})}$ of finite index. This finite index is bounded by size of $\mathfrak{g}'(L)$. 
\end{cor}

\smallskip

\begin{proof} (1) is a direct consequence of Corollary \ref{cor 6.3.7}-(1). The equivalence of the three statements in (2) follows from Lemma \ref{lemma 6.1.1}-(3) and Corollary \ref{cor 6.3.7}-(2).

\medskip

Now we prove that if $l=p$ and $H^{1}(L, \mathfrak{g})^{\vee}$ is a $\Lambda$-submodule of ${\Lambda}^{[L:\mathbb{Q}_p]\cdot \corank_{\Lambda}\mathfrak{g}(\overline{L})}$ of finite index, then this index is bounded by the size of $\mathfrak{g}'(L)$. 

Let $C$ be the kernel of inclusion $H^{1}(L, \mathfrak{g})^{\vee} \hookrightarrow {\Lambda}^{2[L:\mathbb{Q}_p]\cdot \corank_{\Lambda}\mathfrak{g}(\overline{L})}$. For the large enough $n$, we have an exact sequence $0 \rightarrow C \hookrightarrow \frac{H^{1}(L, \mathfrak{g})^{\vee}}{\omega_nH^{1}(L, \mathfrak{g})^{\vee}} \rightarrow \frac{\Lambda}{\omega_n \Lambda}^{2[L:\mathbb{Q}_p]\cdot \corank_{\Lambda}\mathfrak{g}(\overline{L})}$ by the snake lemma. Hence we get an isomorphism $$ C \simeq \frac{H^{1}(L, \mathfrak{g})^{\vee}}{\omega_n H^{1}(L, \mathfrak{g})^{\vee}}[p^{\infty}] $$ for large enough $n$. The last group is the Pontryagin dual of $\left(H^1(L, \mathfrak{g})[\omega_n]\right)_{/div}$.

On the other hand, from a natural exact sequence $$0 \rightarrow \frac{\mathfrak{g}(L)}{\omega_{n}\mathfrak{g}(L)} \rightarrow H^{1}(L, \hat{J}_{n}^{ord}[p^{\infty}]) \rightarrow H^{1}(L, \mathfrak{g})[\omega_{n}] \rightarrow 0,$$ $\left(H^1(L, \mathfrak{g})[\omega_n]\right)_{/div}$ is a homomorphic image of the group $H^{1}(L, \hat{J}_{n}^{ord}[p^{\infty}])_{/div}$, which is the Pontryagin dual of $\hat{J'}_{n}^{ord}[p^{\infty}](L)$ by Corollary \ref{cor 5.2.8}-(1). Hence $|C|$ is bounded by $|\hat{J'}_{n}^{ord}[p^{\infty}](L)|$ for all sufficiently large $n$.
\end{proof}

\smallskip

\begin{cor}\label{cor 6.3.9} If $l=p$ and $\mathfrak{g}'(L)=0$ holds, then we have:
\begin{itemize}
\item $H^{1}(L, \hat{J}_{n}^{ord}[p^{\infty}])$ is a $W$-cofree module of corank $p^{n-1}\cdot[L:\mathbb{Q}_p]\cdot \corank_{\Lambda}\mathfrak{g}(\overline{L})$.
\item $H^{1}(L, \mathfrak{g})$ is a $\Lambda$-cofree module of corank $[L:\mathbb{Q}_p]\cdot \corank_{\Lambda}\mathfrak{g}(\overline{L})$. 
\end{itemize}
\end{cor}

\begin{proof} Since $ \frac{H^{1}(L, \hat{J}_{n}^{ord}[p^{\infty}])}{pH^{1}(L, \hat{J}_{n}^{ord}[p^{\infty}])} \simeq H^{2}(L, \hat{J}_{n}^{ord}[p]) \simeq \hat{J'}_{n}^{ord}[p](L)^{\vee}=0$ by the assumption, $H^{1}(L, \hat{J}_{n}^{ord}[p^{\infty}])$ is a $W$-cofree. The second statement follows from Corollary \ref{cor 6.3.8}-(2). 
\end{proof}

\smallskip

\subsection{$\Lambda$-adic Mordell-Weil groups over $p$-adic fields} \label{sec 6.4}

We study the $\Lambda$-module structure of the Mordell-Weil groups over $p$-adic fields in this subsection. Recall that $\displaystyle \mathfrak{G}(X)=\lim_{\substack{\longleftarrow \\n}}(\frac{X}{\omega_n X}[p^{\infty}])$ for a finitely generated $\Lambda$-module $X$.

\medskip

\begin{prop}\label{prop 6.4.10} Let $L$ be a finite extension of $\mathbb{Q}_p$.

(1) The natural map $\hat{J}^{ord}_{n}(L) \otimes_{\mathbb{Z}_p} \mathbb{Q}_p/\mathbb{Z}_p \rightarrow J^{ord}_{\infty}(L) \otimes_{\mathbb{Z}_p} \mathbb{Q}_p/\mathbb{Z}_p[\omega_n]$ has finite kernel for all $n$ whose orders are bounded as $n$ varies. 

\smallskip

(2) $(J^{ord}_{\infty}(L) \otimes_{\mathbb{Z}_p} \mathbb{Q}_p/\mathbb{Z}_p)^{\vee}$ is $\mathbb{Z}_p$-torsion-free. In particular, $(J^{ord}_{\infty}(L) \otimes_{\mathbb{Z}_p} \mathbb{Q}_p/\mathbb{Z}_p)^{\vee}$ has no non-trivial finite $\Lambda$-submodules.

\smallskip

(3) $\mathfrak{G} \left((J^{ord}_{\infty}(L) \otimes_{\mathbb{Z}_p} \mathbb{Q}_p/\mathbb{Z}_p)^{\vee} \right)=0.$

\smallskip

(4) We have a $\Lambda$-linear injection $$\displaystyle (J^{ord}_{\infty}(L) \otimes_{\mathbb{Z}_p} \mathbb{Q}_p/\mathbb{Z}_p)^{\vee} \hookrightarrow \Lambda^{r} 
\oplus \left(\bigoplus_{n=1}^{t} \frac{\Lambda}{\omega_{b_{n}+1, b_{n}}}\right)$$ with finite cokernel for some integers $r, b_1, \cdots b_n$. 
\end{prop}

\begin{proof} The proof is essentially same as that of \cite[Theorem 2.1.2]{Lee2018}
\end{proof}

\smallskip

\begin{lemma}\label{lemma 6.4.11} Let $0 \rightarrow X \rightarrow Y \rightarrow Z \rightarrow 0$ be an exact sequence of finitely generated $\Lambda$-modules. If $char_{\Lambda}(Z_{\Lambda-\mathrm{tor}})$ is coprime to $char_{\Lambda}(Y_{\Lambda-\mathrm{tor}})$, then we have the following statements: 

\smallskip

(1) The narutal injection $X_{\Lambda-\mathrm{tor}} \hookrightarrow Y_{\Lambda-\mathrm{tor}}$ has finite cokernel. 

\smallskip

(2) If we assume furthermore that $Z$ does not have any non-trivial finite $\Lambda$-submodule, then we have an isomorphism $X_{\Lambda-\mathrm{tor}} \simeq Y_{\Lambda-\mathrm{tor}}$.  Hence $X$ is $\Lambda$-torsion-free if and only if $Y$ is $\Lambda$-torsion-free. 

(3) Under the same assumption of (2), if $X$ is a torsion $\Lambda$-module, then $Z$ is a torsion-free $\Lambda$-module with $\rank_{\Lambda}Y=\rank_{\Lambda}Z$. 

(4) Under the same assumption of (2), if $Y$ is a $\Lambda$-free module, then $X$ is also a $\Lambda$-free module.
\end{lemma}

\begin{proof} From an exact sequence $0 \rightarrow X_{\Lambda-\mathrm{tor}} \rightarrow Y_{\Lambda-\mathrm{tor}} \rightarrow Z_{\Lambda-\mathrm{tor}}$ and the assumption about characteristic ideals, the image of $Y_{\Lambda-\mathrm{tor}} \rightarrow Z_{\Lambda-\mathrm{tor}}$ should be finite, which shows (1). Note that (2) is a direct consequence of (1).

For (3), if $X$ is a torsion $\Lambda$-module, then we have the following exact sequence $$0 \rightarrow X_{\Lambda-\mathrm{tor}} \rightarrow Y_{\Lambda-\mathrm{tor}} \rightarrow Z_{\Lambda-\mathrm{tor}} \rightarrow 0$$ by Lemma \ref{lemma A.0.2}-(1). Since $char_{\Lambda}(Z_{\Lambda-\mathrm{tor}})$ is coprime to $char_{\Lambda}(Y_{\Lambda-\mathrm{tor}})$, we have $char_{\Lambda}(Z_{\Lambda-\mathrm{tor}})=\Lambda$ due to multiplicative property of characteristic ideals. Since $Z$ does not contain any non-trivial finite $\Lambda$-submodule, we have $Z_{\Lambda-\mathrm{tor}}=0$.

Now we prove (4). By the assumptions on $Y$ and $Z$, we have $X[\omega_n]=Y[\omega_n]=0$ and an exact sequence $Z[\omega_n] \hookrightarrow \frac{X}{\omega_nX} \rightarrow \frac{Y}{\omega_nY} $. Here $Z[\omega_n]$ is a free $\mathbb{Z}_p$-module since $Z$ does not have any non-trivial finite $\Lambda$-submodule, and $\frac{Y}{\omega_nY}$ is a $\mathbb{Z}_p$-free module since $Y$ is a $\Lambda$-free. Therefore $\frac{X}{\omega_nX}$ is also a free $\mathbb{Z}_p$-module and this shows that $X$ is a $\Lambda$-free by \cite[Proposition 5.3.19]{neukirch2000cohomology}.
\end{proof}

\medskip

By applying the above lemma to the exact sequence $$0 \rightarrow \left(\frac{H^{1}(L, \mathfrak{g})}{J^{ord}_{\infty}(L) \otimes_{\mathbb{Z}_p} \mathbb{Q}_p/\mathbb{Z}_p}\right)^{\vee} \rightarrow (H^{1}(L, \mathfrak{g}))^{\vee} \rightarrow (J^{ord}_{\infty}(L) \otimes_{\mathbb{Z}_p} \mathbb{Q}_p/\mathbb{Z}_p)^{\vee} \rightarrow 0 ,$$

we obtain the following proposition.

\medskip

\begin{prop}\label{prop 6.4.12} Let $L$ be a finite extension of $\mathbb{Q}_p$.

\smallskip

(1) $ \corank_{\Lambda}\left(J^{ord}_{\infty}(L) \otimes_{\mathbb{Z}_p} \mathbb{Q}_p/\mathbb{Z}_p \right) \geq \frac{1}{2}[L:\mathbb{Q}_p]\cdot \corank_{\Lambda}\mathfrak{g}(\overline{L}) \geq \corank_{\Lambda}\left(\frac{H^{1}(L, \mathfrak{g})}{J^{ord}_{\infty}(L) \otimes_{\mathbb{Z}_p} \mathbb{Q}_p/\mathbb{Z}_p}\right) $.

\smallskip

(2) We have an isomorphism $\left(\frac{H^{1}(L, \mathfrak{g})}{J^{ord}_{\infty}(L) \otimes_{\mathbb{Z}_p} \mathbb{Q}_p/\mathbb{Z}_p}\right)^{\vee}_{\Lambda-\mathrm{tor}} \simeq H^{1}(L, \mathfrak{g})^{\vee}_{\Lambda-\mathrm{tor}}$. 

\smallskip

(3) We have an $\Lambda$-linear injection $\left(\frac{H^{1}(L, \mathfrak{g})}{J^{ord}_{\infty}(L) \otimes_{\mathbb{Z}_p} \mathbb{Q}_p/\mathbb{Z}_p}\right)^{\vee} \hookrightarrow {\Lambda}^{r} \oplus \left(\mathfrak{g}'(L)^{\vee}\right)^{\iota}$ with finite cokernel where $r \leq \frac{1}{2}[L:\mathbb{Q}_p]\cdot \corank_{\Lambda}\mathfrak{g}(\overline{L})$.

\smallskip

(4) $\left(\frac{H^{1}(L, \mathfrak{g})}{J^{ord}_{\infty}(L) \otimes_{\mathbb{Z}_p} \mathbb{Q}_p/\mathbb{Z}_p}\right)^{\vee}$ is $\Lambda$-torsion-free if and only if $H^{1}(L, \mathfrak{g})$ is $\Lambda$-torsion-free if and only if $\mathfrak{g}'(L)$ is finite.

\smallskip

(5) Assume that $H^{1}(L, \mathfrak{g})$ is $\Lambda$-cofree $($For instance, under the assumption of $\mathrm{Corollary}$  \ref{cor 6.3.9}$).$ Then $\frac{H^{1}(L, \mathfrak{g})}{J^{ord}_{\infty}(L) \otimes_{\mathbb{Z}_p} \mathbb{Q}_p/\mathbb{Z}_p}$ is a cofree $\Lambda$-module with $$\mathrm{corank}_{\Lambda}\left(\frac{H^{1}(L, \mathfrak{g})}{J^{ord}_{\infty}(L) \otimes_{\mathbb{Z}_p} \mathbb{Q}_p/\mathbb{Z}_p}\right) \leq \frac{1}{2}[L:\mathbb{Q}_p] \cdot \corank_{\Lambda}\mathfrak{g}(\overline{L}).$$
\end{prop}

\medskip

\begin{proof} For (1), by Lemma \ref{lemma 6.1.2}-(2), it is sufficient to prove the first inequality only. Since $ \hat{J}_{n}^{ord}(L) $ is a $p$-adic Lie group of dimension $\frac{p^{n-1}}{2}\cdot [L:\mathbb{Q}_p]\cdot \corank_{\Lambda}\mathfrak{g}(\overline{L})$ by \cite[Lemma 5.5]{hida2015analytic}, we get (1) from Proposition \ref{prop 6.4.10}-(1).

\medskip
By Proposition \ref{prop 6.4.10}-(2), $(J^{ord}_{\infty}(L) \otimes_{\mathbb{Z}_p} \mathbb{Q}_p/\mathbb{Z}_p)^{\vee}$ has no non-trivial pseudo-null $\Lambda$-submodule. This shows (2) by Lemma \ref{lemma 6.4.11}-(2). (3) follows from (1) and Corollary \ref{cor 6.3.7}-(2). (4) follows from Lemma \ref{lemma 6.4.11}-(2) and Corollary \ref{cor 6.3.8}-(2). (5) is a direct consequence of Lemma \ref{lemma 6.4.11}-(4).
\end{proof}

\medskip

\section{$\Lambda$-adic global cohomology groups}

\smallskip

In this section, in addition to Notation \ref{No} and Convention \ref{Id}, we also assume the following.

\begin{conv}[Global]\label{Gl} We let $K$ be a number field and $S$ be a finite set of places of $K$ containing infinite places and places dividing $Np$. 
\end{conv}

\smallskip

We first mention one basic lemma on $\Lambda$-cotorsionness of $H^{0}$ for our later use. They are analogues of Lemma \ref{lemma 6.1.1}-(1) and hence we omit the proof. (The proof of Lemma \ref{lemma 6.1.1}-(1) still works if we use Mordell-Weil theorem instead of Nagell-Lutz.)

\medskip

\begin{lemma}\label{lemma 7.0.2} $\mathfrak{g}(K)^{\vee}$ is a finitely generated torsion $\Lambda$-module whose characteristic ideal is prime to $\omega_{n}$ for all $n$. In particular, $\frac{\mathfrak{g}(K)}{\omega_{n}\mathfrak{g}(K)}$ has finite bounded order independent of $n$. We have $\displaystyle \lim_{\substack{\longrightarrow \\n}}\frac{\mathfrak{g}(K)}{\omega_{n}\mathfrak{g}(K)}=0$. Moreover $\displaystyle \lim_{\substack{\longleftarrow \\n}}\frac{\mathfrak{g}(K)}{\omega_{n}\mathfrak{g}(K)}$ is isomorphic to the Pontryagin dual of the maximal finite submodule of $\mathfrak{g}(K)^{\vee}$ as $\Lambda$-modules.
\end{lemma}

\smallskip

\subsection{Cofreeness of $H^{2}$} \label{sec 7.1}

\begin{lemma} \label{lemma 7.1.3} Let $A$ be an abelian variety over a number field $K$, and $S$ be a finite set of places of $K$ containing infinite places, places over $p$, and places of bad reductions of $A$. Then we have:

(1) $H^{i}(K^{S}/K, A[p^{\infty}])=0$ for all $i \geq 3$.

(2) $H^2(K^{S}/K, A[p^{\infty}])$ is a $W$-cofree module. In particular, $H^2(K^{S}/K, A[p^{\infty}])=0$ if and only if $H^2(K^{S}/K, A[p^{\infty}])$ has trivial $W$-corank.
\end{lemma}

\smallskip

\begin{proof} Since $\text{Gal}(K^{S}/K)$ has $p$-cohomological dimension 2 ($p$ is odd), we get $H^{i}(K^{S}/K, A[p^{m}])=0$ for all $m$ and for all $i \geq 3$. Passing to the direct limit gives (1). For (2), note that we have an injection $$\frac{H^2(K^{S}/K, A[p^{\infty}])}{p H^2(K^{S}/K, A[p^{\infty}])} \hookrightarrow H^3(K^{S}/K, A[p])=0.$$ Hence $H^2(K^{S}/K, A[p^{\infty}])$ is a $p$-divisible cofinitely generated $W$-module, so $W$-cofree.
\end{proof}

\smallskip

\begin{cor}\label{cor 7.1.4} (1) $H^{i}(K^{S}/K, \hat{J}_{n}^{ord}[p^{\infty}])=H^{i}(K^{S}/K, \mathfrak{g})=0$ for all $i \geq 3$.

(2) $H^{2}(K^{S}/K, \mathfrak{g})$ is a $\Lambda$-cofree module.
\end{cor}

\begin{proof} The first assertion of (1) follows from Lemma \ref{lemma 7.1.3}-(1). Taking the direct limit with respect to $n$ proves the assertion for $\mathfrak{g}$.

For (2), we have an injection $\frac{H^{2}(K^{S}/K, \mathfrak{g})}{\omega_{n}H^{2}(K^{S}/K, \mathfrak{g})} \hookrightarrow H^{3}(K^{S}/K, \hat{J}_{n}^{ord}[p^{\infty}])$ which implies that $\frac{H^{2}(K^{S}/K, \mathfrak{g})}{\omega_{n}H^{2}(K^{S}/K, \mathfrak{g})}=0$ by the previous Lemma \ref{lemma 7.1.3}-(1). On the other hand, ${H^{2}(K^{S}/K, \mathfrak{g})}[\omega_{n}]$ is $W$-cofree since it is a homomorphic image of a $W$-cofree module $H^{2}(K^{S}/K, \hat{J}_{n}^{ord}[p^{\infty}])$. By \cite[Proposition 5.3.19]{neukirch2000cohomology}, $H^{2}(K^{S}/K, \mathfrak{g})$ is $\Lambda$-cofree. 
\end{proof}

\begin{cor}\label{cor 7.1.5} If $H^{2}(K^{S}/K, \mathfrak{g})=0$, then $H^{1}(K^{S}/K, \mathfrak{g})^{\vee}$ has no non-trivial finite $\Lambda$-submodules.
\end{cor}

\begin{proof} 
We have an isomorphism $ H^{1}(K^{S}/K, \mathfrak{g})^{\vee}[\omega_{n}] \simeq H^{2}(K^{S}/K, \hat{J}_{n}^{ord}[p^{\infty}])^{\vee}$ where the last group is $W$-free by Lemma \ref{lemma 7.1.3}-(2). For the large enough $n$, the maximal finite submodule of $ H^{1}(K^{S}/K, \mathfrak{g})^{\vee}$ is contained in $ H^{1}(K^{S}/K, \mathfrak{g})^{\vee}[\omega_{n}]$. Hence it has to vanish since any $W$-free module does not have non-trivial finite submodules. 
\end{proof}

We have the following global Euler characteristic formula of cohomology of Barsotti-Tate groups.

\begin{prop}[Global Euler Characteristic Formula] \label{prop 7.1.6} We have the following two identities:

\begin{itemize}
\item $\corank_{W}H^{1}(K^{S}/K, \hat{J}_{n}^{ord}[p^{\infty}])-\corank_{W}H^{2}(K^{S}/K, \hat{J}_{n}^{ord}[p^{\infty}])=\frac{1}{2}[K:\mathbb{Q}]\cdot p^{n-1}\cdot \corank_{\Lambda}\mathfrak{g}(\overline{\mathbb{Q}}).$
\item $\corank_{\Lambda}H^{1}(K^{S}/K, \mathfrak{g})-\corank_{\Lambda}H^{2}(K^{S}/K, \mathfrak{g})=\frac{1}{2}[K:\mathbb{Q}]\cdot \corank_{\Lambda}\mathfrak{g}(\overline{\mathbb{Q}}).$
\end{itemize}

\smallskip

In particular, $$\corank_{\Lambda}H^{1}(K^{S}/K, \mathfrak{g})=\frac{1}{2}[K:\mathbb{Q}]\cdot \corank_{\Lambda}\mathfrak{g}(\overline{\mathbb{Q}})$$ if and only if $$H^{2}(K^{S}/K, \mathfrak{g})=0.$$
\end{prop}

\medskip

The assertion about the $W$-corank follows from the usual global Euler characteristic formula combined with Theorem \ref{theorem 4.3.11}-(2). The last equivalence follows from Corollary \ref{cor 7.1.4}-(2).

\begin{proof} If we compute the $W$-$\corank$ of the exact sequence $$0 \rightarrow \frac{H^{1}(K^{S}/K, \mathfrak{g})^{\vee}}{\omega_nH^{1}(K^{S}/K, \mathfrak{g})^{\vee}}  \rightarrow H^{1}(K^{S}/K, \hat{J}_{n}^{ord}[p^{\infty}])^{\vee} \rightarrow \mathfrak{g}(K)^{\vee}[\omega_n] \rightarrow 0,$$ the quantity
\begin{align*}
\corank_{W}H^{1}(K^{S}/K, \hat{J}_{n}^{ord}[p^{\infty}])-p^{n-1}\cdot \corank_{\Lambda}H^{1}(K^{S}/K, \mathfrak{g})
\end{align*}
is bounded independent of $n$. Since the same assertion holds for $H^{2}$, we get the assertion about the $\Lambda$-corank.
\end{proof}

\subsection{Application of the Poitou-Tate duality} \label{sec 7.2}

\medskip

In this section, we prove a theorem relating $\Lambda$-coranks of $\sha^{1}_{K}(\mathfrak{g})$ and $\sha^{2}_{K}(\mathfrak{g}')$. Main tool is the Poitou-Tate duality between $\sha^{1}$ and $\sha^{2}$. We state a technical lemma without proof.

\medskip

\begin{lemma}\label{lemma 7.2.7} (1) Let $A$ be a cofinitely generated $W$-module. Then $A[p^{m}]$ and $\frac{A}{p^{m}A}$ are finite modules. For all sufficiently large $m$, $|A[p^m]|=p^{m \cdot \rank_{\mathbb{Z}_p}W \cdot \corank_{W}A+c}$ for some constant $c$ independent of $m$ and $|\frac{A}{p^{m}A}|$ is a constant independent of $m$.

\smallskip

(2) Let $X$ be a cofinitely generated $\Lambda$-module. Then $X[\omega_{n}]$ and $\frac{X}{\omega_{n}X}$ are cofinitely generated $W$-modules. For all sufficiently large $n$, $\corank_{W}X[\omega_{n}]=p^{n \cdot \rank_{\Lambda}X}+c$ for some constant $c$ independent of $n$, and $\corank_{W}\frac{X}{\omega_{n}X}$ is a constant independent of $n$. 

\end{lemma}

Note first that by the local vanishing of $H^{2}$ of the Barsotti-Tate groups (Lemma \ref{lemma 6.1.1}-(2)), we have $$\sha^{2}_{K}(\hat{J}_{n}^{ord}[p^{\infty}])=H^{2}(K^{S}/K, \hat{J}_{n}^{ord}[p^{\infty}]) \quad \mathrm{and} \quad \sha^{2}_{K}(\mathfrak{g})=H^{2}(K^{S}/K, \mathfrak{g}).$$ 

\medskip

\begin{prop}\label{prop 7.2.8} (1) $\rank_{W}(\sha^{1}_{K}(\hat{J}_{n}^{ord}[p^{\infty}])^{\vee})=\rank_{W}(\sha^{2}_{K}(\hat{J'}_{n}^{ord}[p^{\infty}])^{\vee})$. In particular, $\sha^{1}_{K}(\hat{J}_{n}^{ord}[p^{\infty}])$ is finite if and only if $\sha^{2}_{K}(\hat{J'}_{n}^{ord}[p^{\infty}])=0$.

\smallskip

(2) $\rank_{\Lambda}(\sha^{1}(\mathfrak{g})^{\vee})=\rank_{\Lambda}(\sha^{2}(\mathfrak{g}')^{\vee})$. In particular, $\sha^{1}(\mathfrak{g})^{\vee}$ is $\Lambda$-torsion if and only if $\sha^{2}(\mathfrak{g}')=0$.
\end{prop}

The second parts of the assertions (1) and (2) follow from the cofreeness of $H^{2}$. (Lemma \ref{lemma 7.1.3}-(2) and Corollary \ref{cor 7.1.4}-(2))

\begin{proof} (1) We let 
\begin{align*}
&A_{m, n}:=\mathrm{Ker} \left( \frac{\hat{J}_{n}^{ord}[p^{\infty}](K)}{p^{m}\hat{J}_{n}^{ord}[p^{\infty}](K)} \rightarrow \prod_{v \in S} \frac{\hat{J}_{n}^{ord}[p^{\infty}](K_{v})}{p^{m}\hat{J}_{n}^{ord}[p^{\infty}](K_{v})} \right) \\
&B_{m, n}:=\mathrm{Coker} \left( \frac{\hat{J}_{n}^{ord}[p^{\infty}](K)}{p^{m}\hat{J}_{n}^{ord}[p^{\infty}](K)} \rightarrow \prod_{v \in S} \frac{\hat{J}_{n}^{ord}[p^{\infty}](K_{v})}{p^{m}\hat{J}_{n}^{ord}[p^{\infty}](K_{v})} \right)
\end{align*} and
\begin{align*}
&C_{m, n}:=\mathrm{Ker} \left(  \frac{H^{1}(K^{S}/K, \hat{J'}_{n}^{ord}[p^{\infty}])}{p^{m}H^{1}(K^{S}/K, \hat{J'}_{n}^{ord}[p^{\infty}])} \rightarrow \prod_{v \in S} \frac{H^{1}(K_{v}, \hat{J'}_{n}^{ord}[p^{\infty}])}{p^{m}H^{1}(K_{v}, \hat{J'}_{n}^{ord}[p^{\infty}])} \right)\\
&D_{m, n}:=\mathrm{Coker} \left(  \frac{H^{1}(K^{S}/K, \hat{J'}_{n}^{ord}[p^{\infty}])}{p^{m}H^{1}(K^{S}/K, \hat{J'}_{n}^{ord}[p^{\infty}])} \rightarrow \prod_{v \in S} \frac{H^{1}(K_{v}, \hat{J'}_{n}^{ord}[p^{\infty}])}{p^{m}H^{1}(K_{v}, \hat{J'}_{n}^{ord}[p^{\infty}])} \right).
\end{align*} 

\smallskip

For fixed $n$ and sufficiently large $m$, sizes of $A_{m, n}, B_{m, n}, C_{m, n}, D_{m, n}$ are constants by Lemma \ref{lemma 7.2.7}-(1).\\

For $i=0, 1$, we have the following commutative diagram:
\begin{center}
\begin{tikzcd}[column sep=4em, row sep=2em]
\frac{H^{i}(K^S/K, \hat{J}^{ord}_{n}[p^{\infty}])}{p^{m}H^{i}(K^S/K, \hat{J}^{ord}_{n}[p^{\infty}])} \arrow[d] \arrow[r] \arrow[d, "Res"'] \arrow[r, "\hookrightarrow"] & H^{i+1}(K^S/K, \hat{J}^{ord}_{n}[p^{m}]) \arrow[r] \arrow[d]  \arrow[d, "Res"] \arrow[r, "\twoheadrightarrow"] & H^{i+1}(K^S/K, \hat{J}^{ord}_{n}[p^{\infty}])[p^m] \arrow[d]  \arrow[d, "Res"] \\
\displaystyle \prod_{v \in S} \frac{H^{i}(K_v, \hat{J}^{ord}_{n}[p^{\infty}])}{p^{m}H^{i}(K_v, \hat{J}^{ord}_{n}[p^{\infty}])} \arrow[r] \arrow[r, "\hookrightarrow"] & \displaystyle \prod_{\substack{v \in S}}H^{i+1}(K_v, \hat{J}^{ord}_{n}[p^{m}]) \arrow[r] \arrow[r, "\twoheadrightarrow"] & \displaystyle \prod_{\substack{v \in S}}H^{i+1}(K_v, \hat{J}^{ord}_{n}[p^{\infty}])[p^m]
\end{tikzcd}
\end{center}

By the snake lemma, we have the following two exact sequences:
\begin{align*}
0 \rightarrow A_{m, n} \rightarrow \sha^{1}_{K}(\hat{J}_{n}^{ord}[p^{m}]) \rightarrow \sha^{1}_{K}(\hat{J}_{n}^{ord}[p^{\infty}])[p^{m}] \rightarrow B_{m, n}  \\
0 \rightarrow C_{m, n} \rightarrow \sha^{2}_{K}(\hat{J'}_{n}^{ord}[p^{m}]) \rightarrow \sha^{2}_{K}(\hat{J'}_{n}^{ord}[p^{\infty}])[p^{m}] \rightarrow D_{m, n}
\end{align*}

\medskip

By the Poitou-Tate duality (Proposition \ref{prop 5.3.9}), we know that $\sha^{1}_{K}(\hat{J}_{n}^{ord}[p^{m}])$ and $\sha^{2}_{K}(\hat{J'}_{n}^{ord}[p^{m}])$ have same size. Since sizes of the $A_{m, n}, B_{m, n}, C_{m, n}, D_{m, n}$ are independent of $m$ for sufficiently large $m$, Lemma \ref{lemma 7.2.7}-(1) shows $\rank_{W}(\sha^{1}_{K}(\hat{J}_{n}^{ord}[p^{\infty}])^{\vee})=\rank_{W}(\sha^{2}_{K}(\hat{J'}_{n}^{ord}[p^{\infty}])^{\vee})$.\\

(2) We let
\begin{align*}
&E_{n}:=\mathrm{Ker} \left( \frac{\mathfrak{g}(K)}{\omega_{n}\mathfrak{g}(K)} \rightarrow \prod_{v \in S} \frac{\mathfrak{g}(K_{v})}{\omega_{n}\mathfrak{g}(K_{v})} \right)\\
&F_{n}:=\mathrm{Coker} \left( \frac{\mathfrak{g}(K)}{\omega_{n}\mathfrak{g}(K)} \rightarrow \prod_{v \in S} \frac{\mathfrak{g}(K_{v})}{\omega_{n} \mathfrak{g}(K_{v})} \right)
\end{align*} and
\begin{align*}
&G_{n}:=\mathrm{Ker} \left( \frac{H^{1}(K^{S}/K, \mathfrak{g}')}{\omega_{n}H^{1}(K^{S}/K, \mathfrak{g}')} \rightarrow \prod_{v \in S} \frac{H^{1}(K_{v}, \mathfrak{g}')}{\omega_{n}H^{1}(K_{v}, \mathfrak{g}')} \right)\\
&H_{n}:=\mathrm{Coker} \left( \frac{H^{1}(K^{S}/K, \mathfrak{g}')}{\omega_{n}H^{1}(K^{S}/K, \mathfrak{g}')} \rightarrow \prod_{v \in S} \frac{H^{1}(K_{v}, \mathfrak{g}')}{\omega_{n}H^{1}(K_{v}, \mathfrak{g}')} \right).
\end{align*} For sufficiently large $n$, $W$-coranks of the $E_{n}, F_{n}, G_{n}, H_{n}$ are constants by Lemma \ref{lemma 7.2.7}-(2).\\

By the same token with the proof of (1), we have the following two exact sequences:
\begin{align*}
0 \rightarrow E_{n} \rightarrow \sha^{1}_{K}(\hat{J}_{n}^{ord}[p^{\infty}]) \rightarrow \sha^{1}_{K}(\mathfrak{g})[\omega_{n}] \rightarrow F_{n}  \\
0 \rightarrow G_{n} \rightarrow \sha^{2}_{K}(\hat{J'}_{n}^{ord}[p^{\infty}]) \rightarrow \sha^{2}_{K}(\mathfrak{g}')[\omega_{n}] \rightarrow H_{n}
\end{align*}

By (1), we know that $\sha^{1}_{K}(\hat{J}_{n}^{ord}[p^{\infty}])$ and $\sha^{2}_{K}(\hat{J'}_{n}^{ord}[p^{\infty}])$ have same $W$-corank. Since $W$-coranks of $E_{n}, F_{n}, G_{n}, H_{n}$ are independent of $n$ for sufficiently large $n$, Lemma \ref{lemma 7.2.7}-(2) shows that $$\rank_{\Lambda}(\sha^{1}(\mathfrak{g})^{\vee})=\rank_{\Lambda}(\sha^{2}(\mathfrak{g}')^{\vee}).$$
\end{proof}

\section{Structure of $\Lambda$-adic Selmer groups}

In this section, we study the $\Lambda$-adic Selmer groups which are defined as a kernel of certain global-to-local restriction maps. Following \cite{greenberg2000iwasawa}, we show that under mild assumptions, ``Selmer-defining map" is surjective (Corollary \ref{cor 8.2.8}). In the subsection \ref{sec 8.3}, using the twisted Flach pairing, we study the algebraic functional equation of Selmer groups between $(\alpha, \delta, \xi)$-tower and $(\delta, \alpha, \xi')$-tower which is one of the main results of this paper (Theorem \ref{theorem 8.3.10}).

\smallskip

We keep Notation \ref{No}, Convention \ref{Id} and Convention \ref{Gl} in this section.

\subsection{Basic properties} \label{sec 8.1}

We define various Selmer groups first.
\begin{definition}\label{definition 8.1.1}
\begin{enumerate}
\item[(1)] $\displaystyle \mathrm{Sel}_{K, p^m}(\hat{J}_{n}^{ord}):=\mathrm{Ker}(H^{1}(K^{S}/K, \hat{J}_{n}^{ord}[p^{m}]) \rightarrow 
\prod_{v \in S}\frac{H^{1}(K_{v}, \hat{J}_{n}^{ord}[p^{m}])}{ \hat{J}_{n}^{ord}(K_{v})/p^m \hat{J}_{n}^{ord}(K_{v})})$

\item[(2)] $\displaystyle \mathrm{Sel}_{K}(\hat{J}_{n}^{ord}):=\mathrm{Ker}(H^{1}(K^{S}/K, \hat{J}_{n}^{ord}[p^{\infty}]) \rightarrow 
\prod_{v \in S}\frac{H^{1}(K_{v}, \hat{J}_{n}^{ord}[p^{\infty}])}{ \hat{J}_{n}^{ord}(K_{v}) \otimes_{\mathbb{Z}_p} \mathbb{Q}_p/\mathbb{Z}_p})$

\item[(3)] $\displaystyle \mathrm{Sel}_{K}(J_{\infty}^{ord}):=\mathrm{Ker}(H^{1}(K^{S}/K, \mathfrak{g}) \rightarrow 
\prod_{v \in S}\frac{H^{1}(K_{v}, \mathfrak{g})}{ J_{\infty}^{ord}(K_{v}) \otimes_{\mathbb{Z}_p} \mathbb{Q}_p/\mathbb{Z}_p})=\lim_{\substack{\longrightarrow \\ n}}\mathrm{Sel}_{K}(\hat{J}_{n}^{ord})$
\end{enumerate}
\end{definition}

\smallskip

\begin{nota}\label{Selmer} We let $$s_n:\mathrm{Sel}_{K}(\hat{J}_{n}^{ord}) \rightarrow \mathrm{Sel}_{K}(J_{\infty}^{ord})[\omega_{n}]$$ be the natural restriction map. We also define $$s'_n:\mathrm{Sel}_{K}(\hat{J'}_{n}^{ord}) \rightarrow \mathrm{Sel}_{K}(J_{\infty}^{' ord})[\omega_{n}].$$ By Hida's control theorem (Theorem \ref{Theorem H}), $\mathrm{Coker}(s_n)$ and $\mathrm{Coker}(s'_n)$ are finite for all $n \geq 2$.
\end{nota}

\smallskip

\begin{remark}\label{remark 8.1.3} (1) $\mathrm{Sel}_{K, p^m}(\hat{J}_{n}^{ord})$ is finite for all $m, n$.

\smallskip

(2) We have an exact sequence  $$0 \rightarrow \frac{\hat{J}_{n}^{ord}[p^{\infty}](K)}{p^{m}\hat{J}_{n}^{ord}[p^{\infty}](K)} \rightarrow \mathrm{Sel}_{K, p^m}(\hat{J}_{n}^{ord}) \rightarrow \mathrm{Sel}_{K}(\hat{J}_{n}^{ord})[p^{m}] \rightarrow 0$$ for all $m, n$. In particular, $ \frac{ |  \mathrm{Sel}_{K, p^m}(\hat{J}_{n}^{ord}) | }{p^{m \cdot \corank_{W} \mathrm{Sel}_{K}(\hat{J}_{n}^{ord})}} $ becomes stable as $m$ goes to infinity.

\smallskip

(3) By definition of the Selmer groups, we have an injection $$\mathrm{Ker}(s_n) \hookrightarrow \mathrm{Ker} \left(H^{1}(K^S/K,\hat{J}_{n}^{ord}[p^{\infty}] ) \rightarrow H^{1}(K^S/K, \mathfrak{g})[\omega_n]\right) \simeq \frac{\mathfrak{g}(K)}{\omega_n\mathfrak{g}(K)}.$$ Hence $\mathrm{Ker}(s_n)$ is finite and bounded independent of $n$ by Lemma \ref{lemma 7.0.2}. Moreover, we have $$\displaystyle \lim_{\substack {\longrightarrow \\ n}}\mathrm{Ker}(s_n)=\lim_{\substack {\longrightarrow \\ n}}\mathrm{Coker}(s_n)=0.$$ We also have an injection
$$\displaystyle \lim_{\substack {\longleftarrow \\ n}}\mathrm{Ker}(s_n) \hookrightarrow N^{\vee}$$ by Lemma \ref{lemma 7.0.2}, where $N$ is the maximal finite $\Lambda$-submodule of $\mathfrak{g}(K)^{\vee}$.
\end{remark}

\smallskip

As a consequence of the last remark, we have the following formula relating the corank of $p$-adic Selmer groups of $(\alpha, \delta, \xi)$-tower and $(\delta, \alpha, \xi')$-tower.

\medskip

\begin{cor}\label{cor 8.1.4} (1) (Greenberg-Wiles formula) The value $\frac{\mid \mathrm{Sel}_{K, p^m}(\hat{J}_{n}^{ord}) \mid}{\mid \mathrm{Sel}_{K, p^m}(\hat{J'}_{n}^{ord}) \mid}$ becomes stationary as $m$ goes to infinity.

(2) $\corank_{W} \mathrm{Sel}_{K}(\hat{J}_{n}^{ord})=\corank_{W} \mathrm{Sel}_{K}(\hat{J'}_n^{ord})$ for all $(\alpha, \delta, \xi)$ and $n$. Hence we have $$\corank_{W} \left( \mathrm{Sel}_{K}(J_{\infty}^{ord})[\omega_{n}]\right)=\corank_{W} \left(\mathrm{Sel}_{K}(J_{\infty}^{' ord})[\omega_{n}]\right).$$\end{cor}

\begin{proof} (1) follows from the Greenberg-Wiles formula \cite[Theorem 2.19]{darmon1995fermat} combined with Proposition \ref{prop 5.2.7}-(2). The first statement of (2) is a direct consequence of (1) and Remark \ref{remark 8.1.3}-(2). The second statement of (2) follows from Hida's control theorem (Theorem \ref{Theorem H}).
\end{proof}

\smallskip

\subsection{Surjectivity of global-to-local restriction map} \label{sec 8.2}

Recall that $\mathrm{Sel}_{K}(J_{\infty}^{ord})$ is defined as a kernel of global-to-local restriction map $\displaystyle H^{1}(K^{S}/K, \mathfrak{g}) \rightarrow 
\prod_{v \in S}\frac{H^{1}(K_{v}, \mathfrak{g})}{ J_{\infty}^{ord}(K_{v}) \otimes_{\mathbb{Z}_p} \mathbb{Q}_p/\mathbb{Z}_p}$. We let $C_{\infty}(K)$ be the \textbf{cokernel of this restriction map.} In this subsection, we show that this group is ``small" under mild assumptions. (Corollary \ref{cor 8.2.8}).\\

We first construct various exact sequences induced from the Poitou-Tate sequence. For a finitely or a cofinitely generated $\mathbb{Z}_p$-module $M$, we let $\displaystyle T_{p}M:=\lim_{\substack {\longleftarrow \\ n}}M[p^n] \simeq \mathrm{Hom}_{\mathbb{Z}_p}(\mathbb{Q}_p/\mathbb{Z}_p,M)$. Note that for $M$ with finite cardinality, we have $T_pM=0$ and $\mathrm{Ext}^{1}_{\mathbb{Z}_p}(\mathbb{Q}_p/\mathbb{Z}_p,M) \simeq M$ where the isomorphism is functorial in $M$. Hence the exact sequence $0 \rightarrow Q \rightarrow R \rightarrow U \rightarrow V \rightarrow 0$ of cofinitely generated $\mathbb{Z}_p$-modules with finite $Q$ induces another exact sequence $0 \rightarrow T_pR \rightarrow T_pU \rightarrow T_pV$. 

\smallskip

\begin{prop}\label{prop 8.2.5} (1) We have the following four exact sequences:
\begin{align}
\label{e_12}\tag{W} \displaystyle 0 \rightarrow \mathrm{Sel}_{K}(J_{\infty}^{ord}) \rightarrow H^{1}(K^{S}/K, \mathfrak{g}) \rightarrow  
\prod_{v \in S} \frac{H^{1}(K_{v}, \mathfrak{g}) }{J_{\infty}^{ord}(K_{v}) \otimes_{\mathbb{Z}_p} \mathbb{Q}_p/\mathbb{Z}_p} \rightarrow C_{\infty}(K) \rightarrow 0\\
\label{e_13}\tag{X}0 \rightarrow C_{\infty}(K) \rightarrow \left(\lim_{\substack{\longleftarrow \\ n}} \lim_{\substack{\longleftarrow \\ m}} \mathrm{Sel}_{K, p^m}(\hat{J'}_{n}^{ord})\right)^{\vee} \rightarrow H^{2}(K^{S}/K, \mathfrak{g}) \rightarrow 0\\
\label{e_14}\tag{Y}\displaystyle 0 \rightarrow \mathfrak{F}\left(\mathfrak{g}'(K)^{\vee}\right) \rightarrow \lim_{\substack{\longleftarrow \\ n}} \lim_{\substack{\longleftarrow \\ m}} \mathrm{Sel}_{K, p^m}(\hat{J'}_{n}^{ord}) \rightarrow \lim_{\substack {\longleftarrow \\ n}}T_{p}\mathrm{Sel}_{K}(\hat{J'}_{n}^{ord}) \rightarrow 0 \\
\label{e_15}\tag{Z}0 \rightarrow \lim_{\substack {\longleftarrow \\ n}}T_{p}\mathrm{Sel}_{K}(\hat{J'}_{n}^{ord}) \rightarrow \mathrm{Hom}_{\Lambda}(\mathrm{Sel}_{K}(J_{\infty}^{' ord})^{\vee}, \Lambda) \rightarrow \lim_{\substack {\longleftarrow \\ n}}T_{p}\mathrm{Coker}(s'_n)
\end{align}

(2) For all places $v$ of $K$ dividing $p$, let $$\epsilon_{v}:=\corank_{\Lambda}(J_{\infty}^{ord}(K_{v}) \otimes_{\mathbb{Z}_p} \mathbb{Q}_p/\mathbb{Z}_p)-\frac{1}{2}[K_v:\mathbb{Q}_p]\cdot \corank_{\Lambda}\mathfrak{g}(\overline{K_v}).$$  Note that $\epsilon_{v}$ is non-negative by Proposition \ref{prop 6.4.12}-(1). We have the following identity:$$\displaystyle \corank_{\Lambda}\left(\mathrm{Sel}_{K}(J_{\infty}^{ord})\right)=\corank_{\Lambda}\left(H^{2}(K^{S}/K, \mathfrak{g})\right)+\corank_{\Lambda}\left(C_{\infty}(K)\right)+\sum_{\substack{v \mid p}}\epsilon_{v}.$$
\end{prop}

\smallskip

\begin{proof} By \cite[page 11]{coates2000galois}, we have an exact sequence 
\begin{align*}
0 \rightarrow \mathrm{Sel}_{K}(\hat{J}_{n}^{ord}) \rightarrow H^{1}(K^S/K, \hat{J}^{ord}_{n}[p^{\infty}]) \rightarrow \prod_{v \in S} \frac{H^{1}(K_{v}, \hat{J}^{ord}_{n}[p^{\infty}]) }{\hat{J}_{n}^{ord}(K_{v}) \otimes_{\mathbb{Z}_p} \mathbb{Q}_p/\mathbb{Z}_p} \\
 \rightarrow \left(\lim_{\substack{\longleftarrow \\ m}} \mathrm{Sel}_{K, p^m}(\hat{J'}_{n}^{ord})\right)^{\vee} \rightarrow H^{2}(K^S/K, \hat{J}^{ord}_{n}[p^{\infty}]) \rightarrow 0.
\end{align*} Taking the direct limit with respect to $n$ proves \ref{e_12} and \ref{e_13}.\\

On the other hand, by \cite[Lemma 1.8]{coates2000galois}, we have an exact sequence $$\displaystyle 0 \rightarrow \hat{J'}_{n}^{ord}[p^{\infty}](K) \rightarrow  \lim_{\substack{\longleftarrow \\ m}} \mathrm{Sel}_{K, p^m}(\hat{J'}_{n}^{ord}) \rightarrow T_{p}\mathrm{Sel}_{K}(\hat{J'}_{n}^{ord}) \rightarrow 0.$$ Taking direct limit with respect to $n$ proves \ref{e_14}. \\

\smallskip

Since $\mathrm{Ker}(s'_n)$ is finite, the exact sequence $$0 \rightarrow \mathrm{Ker}(s'_n) \rightarrow \mathrm{Sel}_{K}(\hat{J'}_{n}^{ord})
 \rightarrow \mathrm{Sel}_{K}(J_{\infty}^{' ord})[\omega_{n}] \rightarrow  \mathrm{Coker}(s'_n) \rightarrow 0$$ induces another exact sequence $$0 \rightarrow T_{p}\mathrm{Sel}_{K}(\hat{J'}_{n}^{ord}) \rightarrow T_{p}(\mathrm{Sel}_{K}(J_{\infty}^{' ord})[\omega_{n}]) \rightarrow T_{p}\mathrm{Coker}(s'_n)$$ by the remark mentioned before this proposition. Taking projective limit with respect to $n$ combined with Lemma \ref{lemma A.0.3}-(2) in the appendix proves \ref{e_15}.
 
\medskip

Lastly, by Proposition \ref{prop 7.1.6}-(2) and Lemma \ref{lemma 6.1.2}-(2) we have:
\begin{align*}
&\corank_{\Lambda}\left(H^{1}(K^{S}/K, \mathfrak{g})\right)=\frac{1}{2}[K:\mathbb{Q}] \cdot \corank_{\Lambda}\mathfrak{g}(\overline{K})+\corank_{\Lambda}\left(H^{2}(K^{S}/K, \mathfrak{g})\right)\\
&\corank_{\Lambda}\left(H^{1}(K_{v}, \mathfrak{g})\right)=0 \quad (v \nmid p).
\end{align*}
Hence (2) follows from the calculation of $\Lambda$-coranks of the four terms in the sequence \ref{e_12}.
\end{proof}

\smallskip

\begin{cor} \label{cor 8.2.6} (1) $\displaystyle \lim_{\substack {\longleftarrow \\ n}}T_{p}\mathrm{Sel}_{K}(\hat{J'}_{n}^{ord})$ is a torsion-free $\Lambda$-module. Hence we have a natural isomorphism $$\displaystyle \mathfrak{F}\left(\mathfrak{g}'(K)^{\vee}\right) \simeq \left(\lim_{\substack{\longleftarrow \\ n}} \lim_{\substack{\longleftarrow \\ m}} \mathrm{Sel}_{K, p^m}(\hat{J'}_{n}^{ord})\right)_{\Lambda-\mathrm{tor}}$$
and these groups vanish if $\mathfrak{g}'(K)$ is finite. 

\smallskip

(2) $\displaystyle \lim_{\substack {\longleftarrow \\ n}}T_{p}\mathrm{Sel}_{K}(\hat{J'}_{n}^{ord})=0$ if and only if $\displaystyle \lim_{\substack{\longleftarrow \\ n}} \lim_{\substack{\longleftarrow \\ m}} \mathrm{Sel}_{K, p^m}(\hat{J'}_{n}^{ord})$ is a torsion $\Lambda$-module. If this is the case, $H^{2}(K^{S}/K, \mathfrak{g})=0$ and we have isomorphisms $$\displaystyle C_{\infty}(K)^{\vee} \simeq \lim_{\substack{\longleftarrow \\ n}} \lim_{\substack{\longleftarrow \\ m}} \mathrm{Sel}_{K, p^m}(\hat{J'}_{n}^{ord}) \simeq \mathfrak{F}\left(\mathfrak{g}'(K)^{\vee}\right)$$
\end{cor}

\begin{proof} We prove (1) first. Since $\mathrm{Hom}_{\Lambda}(\mathrm{Sel}_{K}(J_{\infty}^{' ord})^{\vee}, \Lambda)$ is a free $\Lambda$-module, $\displaystyle \lim_{\substack {\longleftarrow \\ n}}T_{p}\mathrm{Sel}_{K}(\hat{J'}_{n}^{ord})$ is a torsion-free $\Lambda$-module by the sequence \ref{e_15}. The second and the third assertion of (1) follows from the sequence \ref{e_14} and Lemma \ref{lemma A.0.3}-(1).

\medskip

The first statement of (2) follows directly from (1) and the fact that $\mathfrak{F}\left(\mathfrak{g}'(K)^{\vee}\right)$ is a torsion $\Lambda$-module. By the sequence \ref{e_13}, we have an injection $$\displaystyle H^{2}(K^{S}/K, \mathfrak{g})^{\vee} \hookrightarrow \lim_{\substack{\longleftarrow \\ n}} \lim_{\substack{\longleftarrow \\ m}} \mathrm{Sel}_{K, p^m}(\hat{J'}_{n}^{ord})$$ where $H^{2}(K^{S}/K, \mathfrak{g})^{\vee}$ is a $\Lambda$-free by the Corollary \ref{cor 7.1.4}-(2). 

Hence if $\displaystyle \lim_{\substack{\longleftarrow \\ n}} \lim_{\substack{\longleftarrow \\ m}} \mathrm{Sel}_{K, p^m}(\hat{J'}_{n}^{ord})$ is a torsion $\Lambda$-module, then $H^{2}(K^{S}/K, \mathfrak{g})=0$. Hence from the sequence \ref{e_13}, we get an isomorphism $$\displaystyle C_{\infty}(K)^{\vee} \simeq \lim_{\substack{\longleftarrow \\ n}} \lim_{\substack{\longleftarrow \\ m}} \mathrm{Sel}_{K, p^m}(\hat{J'}_{n}^{ord}).$$\end{proof}

\begin{prop}\label{prop 8.2.7} If $\mathrm{Sel}_{K}(J_{\infty}^{ord}) $ is a $\Lambda$-cotorsion module, then we have:

\smallskip

(1) $\sha^{1}_{K}(\mathfrak{g})$ is a cotorsion $\Lambda$-module.

\smallskip

(2) $H^{2}(K^{S}/K, \mathfrak{g})=0$. Hence $H^{1}(K^{S}/K, \mathfrak{g})^{\vee}$ has no non-trivial finite $\Lambda$-submodules by Corollary \ref{cor 7.1.5}. 

\smallskip

(3) $C_{\infty}(K)^{\vee}$ is a $\Lambda$-torsion module isomorphic to $\mathfrak{F}\left(\mathfrak{g}'(K)^{\vee}\right)$.

\smallskip

(4) For all places $v$ dividing $p$, we have 
\begin{align*}
&\corank_{\Lambda}\left(J_{\infty}^{ord}(K_{v}) \otimes_{\mathbb{Z}_p} \mathbb{Q}_p/\mathbb{Z}_p\right)=\frac{1}{2}[K_v:\mathbb{Q}_p] \cdot \corank_{\Lambda}\mathfrak{g}(\overline{K})
\end{align*}
\end{prop}

\begin{proof} The assertion (1) is obvious since $\sha^{1}_{K}(\mathfrak{g})$ is a $\Lambda$-submodule of  $\mathrm{Sel}_{K}(J_{\infty}^{ord})$. By the corank computation in Proposition \ref{prop 8.2.5}-(2), we have $$\corank_{\Lambda}\left(H^{2}(K^{S}/K, \mathfrak{g})\right)=\corank_{\Lambda}\left(C_{\infty}(K)\right)=\epsilon_{v}=0$$ for all $v$ dividing $p$. Hence we have (2) since $H^{2}(K^{S}/K, \mathfrak{g})$ is a cofree $\Lambda$-module. 

From the sequence \ref{e_13} in Proposition \ref{prop 8.2.5}-(1) and the above corank identity, $\displaystyle \lim_{\substack{\longleftarrow \\ n}} \lim_{\substack{\longleftarrow \\ m}} \mathrm{Sel}_{K, p^m}(\hat{J'}_{n}^{ord})$ is a torsion $\Lambda$-module. Now (3) follows from an exact sequence \ref{e_13} combined with Corollary \ref{cor 8.2.6}-(2).
\end{proof}

\medskip

We have the following corollary about the surjectivity of the ``Selmer-defining map", which is an analogue of \cite[Proposition 2.1]{greenberg2000iwasawa} and \cite[Proposition 2.5]{greenberg2000iwasawa}.

\medskip

\begin{cor}\label{cor 8.2.8} Suppose that $\mathrm{Sel}_{K}(J_{\infty}^{ord}) $ is a $\Lambda$-cotorsion.

(1) If $\mathfrak{g}'(K)$ is finite, then we have $C_{\infty}(K)=0$. 

\smallskip

(2) If we furthermore assume that $\mathfrak{g}'(K_{v})$ is finite for all $v \in S$ and $\mathfrak{g}'(K_{v})=0$ for all places $v | p$, then $\mathrm{Sel}_{K}(J_{\infty}^{ord})^{\vee} $ has no non-trivial pseudo-null $\Lambda$-submodules. 
\end{cor}

\begin{proof} By the Proposition \ref{prop 8.2.7}-(3), we have $C_{\infty}(K) \simeq \mathfrak{F}\left(\mathfrak{g}'(K)^{\vee}\right)\simeq 0$ since $\mathfrak{g}'(K)$ is finite.\\

For the second part, by Corollary \ref{cor 6.3.8}-(1),  Proposition \ref{prop 6.4.12}-(5) and Proposition \ref{prop 8.2.7}-(2), we have :

\begin{itemize}
\item $H^{1}(K_v, \mathfrak{g})=0$ for all $v$ in $S$ not dividing $p$. 

\item $\frac{H^{1}(K_{v}, \mathfrak{g})}{ J_{\infty}^{ord}(K_{v}) \otimes_{\mathbb{Z}_p} \mathbb{Q}_p/\mathbb{Z}_p}$ is a cofree $\Lambda$-module for all $v$ dividing $p$.

\item $H^{1}(K^{S}/K, \mathfrak{g})^{\vee}$ has no non-trivial finite $\Lambda$-submodules.
\end{itemize}

\medskip

Since $C_{\infty}(K)=0$ by the first part, we have an exact sequence $$\displaystyle 0 \rightarrow (\prod_{v \mid p}\frac{H^{1}(K_{v}, \mathfrak{g})}{ J_{\infty}^{ord}(K_{v}) \otimes_{\mathbb{Z}_p} \mathbb{Q}_p/\mathbb{Z}_p})^{\vee} \rightarrow H^{1}(K^{S}/K, \mathfrak{g})^{\vee} \rightarrow \mathrm{Sel}_{K}(J_{\infty}^{ord})^{\vee} \rightarrow 0,$$ where the first term is a $\Lambda$-free as we mentioned above. Now \cite[Lemma 2.6]{greenberg2000iwasawa} shows the desired assertion.
\end{proof}

\subsection{Algebraic functional equation of $\Lambda$-adic Selmer groups} \label{sec 8.3}

\smallskip

We first mention technical lemmas without proof.

\smallskip

\begin{lemma}\label{lemma 8.3.9} (1) Let $A$ and $B$ be finitely generated $\Lambda$-torsion modules. If there are $\Lambda$-linear maps $\phi:A \rightarrow B$ and $\psi:B \rightarrow A$ with finite kernels, then $A$ and $B$ are pseudo-isomorphic.

\smallskip

(2) Let $X$ and $Y$ be finitely generated $\Lambda$-modules. Suppose that $\rank_{\mathbb{Z}_p}\frac{X}{\omega_{n}X}=\rank_{\mathbb{Z}_p}\frac{Y}{\omega_{n}Y}$ holds for all $n$, and that there are two $\Lambda$-linear maps $ \mathfrak{G}(X) \rightarrow \mathfrak{F}(Y), \mathfrak{G}(Y) \rightarrow \mathfrak{F}(X) $ with finite kernels. Then $E(X)$ and $E(Y^{\iota})$ are isomorphic as $\Lambda$-modules. 
\end{lemma}

\smallskip

Now we compare $E\left(\mathrm{Sel}_{K}(J_{\infty}^{ord})^{\vee}\right)$ and $E\left(\mathrm{Sel}_{K}(J_{\infty}^{' ord})^{\vee}\right)$.

\begin{theorem}\label{theorem 8.3.10} We have an isomorphism $$E\left(\mathrm{Sel}_{K}(J_{\infty}^{ord})^{\vee}\right) \simeq E\left(\mathrm{Sel}_{K}(J_{\infty}^{' ord})^{\vee}\right)^{\iota}$$ of $\Lambda$-modules.
\end{theorem}

\begin{proof} By Lemma \ref{lemma 8.3.9}, it is sufficient to show the following two statements:

\begin{itemize}
\item $\corank_{W}\left(\mathrm{Sel}_{K}(J_{\infty}^{ord})[\omega_{n}]\right)=\corank_{W}\left(\mathrm{Sel}_{K}(J_{\infty}^{' ord})[\omega_{n}]\right)$ for all $n$.
\item There is a $\Lambda$-linear map $\mathfrak{G}\left(\mathrm{Sel}_{K}(J_{\infty}^{' ord})^{\vee}\right) \rightarrow \mathfrak{F}\left(\mathrm{Sel}_{K}(J_{\infty}^{ord})^{\vee}\right)$ with finite kernel. 
\end{itemize}

\medskip

The first statement is Corollary \ref{cor 8.1.4}-(2). For the second statement, consider an exact sequence $$0 \rightarrow \mathrm{Ker}(s_n) \rightarrow \mathrm{Sel}_{K}(\hat{J}_{n}^{ord}) \rightarrow \mathrm{Sel}_{K}(J_{\infty}^{ord})[\omega_{n}] \rightarrow  \mathrm{Coker}(s_n) \rightarrow 0.$$ By Lemma \ref{lemma A.0.2}-(3), we have $$ \mathrm{Ker}(s_n) \rightarrow \mathrm{Sel}_{K}(\hat{J}_{n}^{ord})_{/div}
 \rightarrow \left(\mathrm{Sel}_{K}(J_{\infty}^{ord})[\omega_{n}]\right)_{/div} \rightarrow  \mathrm{Coker}(s_n) \rightarrow 0.$$ By taking the direct limit and the Pontryagin dual, we get \begin{align}
 \lim_{\substack{\longleftarrow \\n}}\left(\mathrm{Sel}_{K}(\hat{J}_{n}^{ord})_{/div}\right)^{\vee} \simeq \mathfrak{G}\left(\mathrm{Sel}_{K}(J_{\infty}^{ord})^{\vee}\right).
\end{align}

\medskip

Now we look at $(\delta, \alpha, \xi')$-tower. By the same token, we get $$ \mathrm{Ker}(s'_n) \rightarrow \mathrm{Sel}_{K}(\hat{J'}_{n}^{ord})_{/div}
 \rightarrow \left(\mathrm{Sel}_{K}(J_{\infty}^{' ord})[\omega_{n}]\right)_{/div} \rightarrow  \mathrm{Coker}(s'_n) \rightarrow 0.$$ 
 
Note that all the terms in this sequence are finite. By taking projective limit, we get the following exact sequence:

\begin{align}\displaystyle \lim_{\substack{\longleftarrow \\n}}\mathrm{Ker}(s'_n) \rightarrow \lim_{\substack{\longleftarrow \\n}}\mathrm{Sel}_{K}(\hat{J'}_{n}^{ord})_{/div}
 \rightarrow \lim_{\substack{\longleftarrow \\n}}\left(\mathrm{Sel}_{K}(J_{\infty}^{' ord})[\omega_{n}]\right)_{/div}.
\end{align}

We now analyze the three terms in the above sequence (2).

\begin{itemize}
\item $\displaystyle \lim_{\substack{\longleftarrow \\n}}\mathrm{Ker}(s'_n)$ is finite by Remark \ref{remark 8.1.3}-(3).
\smallskip
\item By the perfectness of twisted Flach pairing (Proposition \ref{prop 5.4.11}-(4)), the middle term in (2) is isomorphic to $\displaystyle \lim_{\substack{\longleftarrow \\n}}\left(\mathrm{Sel}_{K}(\hat{J}_{n}^{ord})_{/div}\right)^{\vee}$, and this group is isomorphic to $\mathfrak{G}\left(\mathrm{Sel}_{K}(J_{\infty}^{ord})^{\vee}\right)$ as $\Lambda$-modules by the (1) above.
\smallskip
\item The last term in (2) is $\mathfrak{F}\left(\mathrm{Sel}_{K}(J_{\infty}^{' ord})^{\vee}\right)$ by the definition of the functor $\mathfrak{F}$.
\end{itemize}

Therefore, there is a $\Lambda$-linear map $\mathfrak{G}\left(\mathrm{Sel}_{K}(J_{\infty}^{ord})^{\vee}\right) \rightarrow \mathfrak{F}\left(\mathrm{Sel}_{K}(J_{\infty}^{' ord})^{\vee}\right)$ with the finite kernel, which shows the desired assertion.
\end{proof}

\smallskip

\begin{remark} \label{remark 8.3.11}
If we consider the Tate-Shafarevich groups instead of the Selmer groups, under the analogous assumptions (control of the Tate-Shafarevich groups), we get a pseudo-isomorphism between $$\mathfrak{G}\left(\sha^{1}_{K}(J_{\infty}^{ord})^{\vee}\right) \quad \mathrm{and} \quad\mathfrak{G}\left(\sha^{1}_{K}(J_{\infty}^{' ord})^{\vee}\right)^{\iota}.$$ The proof uses the (twisted) Cassels-Tate paiting, instead of Flach's one. The reason why we can only compare the values of $\mathfrak{G}$ for $\sha^1$ is that we do not have an analogue of the Greenberg-Wiles formula for the $\sha^1$.
\end{remark}

\smallskip

\section{Global $\Lambda$-adic Mordell-Weil group and $\Lambda$-adic Tate-Shafarevich groups}

\medskip

We keep Notation \ref{No}, Convention \ref{Id} and Convention \ref{Gl} in this section.

\subsection{$\Lambda$-adic Mordell-Weil groups} \label{sec 9.1}

We state an analogue of Proposition \ref{prop 6.4.10} whose proofs can be found in \cite[Theorem 2.1.2]{Lee2018}.

\begin{prop}\label{prop 9.1.1} (1) The natural map $\hat{J}^{ord}_{n}(K) \otimes_{\mathbb{Z}_p} \mathbb{Q}_p/\mathbb{Z}_p \rightarrow J^{ord}_{\infty}(K) \otimes_{\mathbb{Z}_p} \mathbb{Q}_p/\mathbb{Z}_p[\omega_n]$ has finite kernel for all $n$ whose orders are bounded as $n$ varies. 

\smallskip

(2) $(J^{ord}_{\infty}(K) \otimes_{\mathbb{Z}_p} \mathbb{Q}_p/\mathbb{Z}_p)^{\vee}$ is $\mathbb{Z}_p$-torsion-free. In particular, $(J^{ord}_{\infty}(K) \otimes_{\mathbb{Z}_p} \mathbb{Q}_p/\mathbb{Z}_p)^{\vee}$ has no non-trivial finite $\Lambda$-submodules.

\smallskip

(3) $\mathfrak{G}\left((J^{ord}_{\infty}(K) \otimes_{\mathbb{Z}_p} \mathbb{Q}_p/\mathbb{Z}_p)^{\vee}\right)=0$.

\smallskip

(4) We have a $\Lambda$-linear injection $$\displaystyle (J^{ord}_{\infty}(K) \otimes_{\mathbb{Z}_p} \mathbb{Q}_p/\mathbb{Z}_p)^{\vee} \hookrightarrow \Lambda^{r} 
\oplus\left(\bigoplus_{n=1}^{t} \frac{\Lambda}{\omega_{b_{n}+1, b_{n}}}\right)$$ with finite cokernel for some integers $r, b_1, \cdots b_n$. 
\end{prop}

\begin{remark}\label{remark 9.1.2} (4) can be regarded as a $\Lambda$-adic analogue of the fact that $\hat{J}^{ord}_{n}(K) \otimes_{\mathbb{Z}_p} \mathbb{Q}_p/\mathbb{Z}_p$ is $W$-cofree.
\end{remark}

\smallskip

\subsubsection{Control and cotorsionness of $\Lambda$-adic Mordell-Weil group} \label{sec 9.1}
Next we show that the global Mordell-Weil groups are ``well-controlled" when $\Lambda$-adic Mordell-Weil group is a \emph{cotorsion}.

\smallskip

\begin{lemma}\label{lemma 9.1.3} (1) The natural map $\hat{J}^{ord}_{n}(K) \otimes_{\mathbb{Z}_p} \mathbb{Q}_p/\mathbb{Z}_p \rightarrow \hat{J}^{ord}_{n+1}(K) \otimes_{\mathbb{Z}_p} \mathbb{Q}_p/\mathbb{Z}_p$ has finite kernel for all $n$.

\smallskip

(2) We always have $\rank_{W}\hat{J}^{ord}_{n}(K) \leq \rank_{W}\hat{J}^{ord}_{n+1}(K)$. Moreover, equality holds if and only if the natural map in (1) is surjective.
\end{lemma}

\begin{proof} (1) is a direct consequence of Proposition \ref{prop 9.1.1}-(1). The first inequality of (2) follows from (1). The second part of (2) follows from the fact that $\hat{J}^{ord}_{n+1}(K) \otimes_{\mathbb{Z}_p} \mathbb{Q}_p/\mathbb{Z}_p$ is a cofree $W$-module.
\end{proof}

\smallskip

\begin{theorem}\label{theorem 9.1.4} The following conditions are equivalent:

(1) The natural inclusion $J^{ord}_{\infty}(K) \otimes_{\mathbb{Z}_p} \mathbb{Q}_p/\mathbb{Z}_p[\omega_n] \hookrightarrow J^{ord}_{\infty}(K) \otimes_{\mathbb{Z}_p} \mathbb{Q}_p/\mathbb{Z}_p[\omega_{n+1}]$ is an isomorphism for almost all $n$.

\smallskip

(2) $\rank_{W}\hat{J}^{ord}_{n}(K)$ stabilizes as $n$ goes to infinity.

\medskip

If this equivalence holds, the natural map $\hat{J}^{ord}_{n}(K) \otimes_{\mathbb{Z}_p} \mathbb{Q}_p/\mathbb{Z}_p \rightarrow J^{ord}_{\infty}(K) \otimes_{\mathbb{Z}_p} \mathbb{Q}_p/\mathbb{Z}_p[\omega_n]$ is surjective except for finitely many $n$ and $$\rank_{W}\hat{J}^{ord}_{n}(K)=\lambda \left((J^{ord}_{\infty}(K) \otimes_{\mathbb{Z}_p} \mathbb{Q}_p/\mathbb{Z}_p)^{\vee}\right)$$ for all sufficiently large $n$.
\end{theorem}

\smallskip

\begin{proof} Proof is almost identical as that of \cite[Theorem 2.1.5]{Lee2018}.
\end{proof}

\subsection{$\Lambda$-adic Tate-Shafarevich groups} \label{sec 9.2}

\medskip

Define a $\Lambda$-adic Tate-Shafarevich group by $\displaystyle \sha^{1}_{K}(J_{\infty}^{ord}):=\lim_{\substack{\longrightarrow \\ n}}\sha^{1}_{K}(\hat{J}_{n}^{ord})$ which fits into a natural exact sequence $$0 \rightarrow J^{ord}_{\infty}(K) \otimes_{\mathbb{Z}_p} \mathbb{Q}_p/\mathbb{Z}_p \rightarrow \mathrm{Sel}_{K}(J_{\infty}^{ord}) \rightarrow \sha^{1}_{K}(J_{\infty}^{ord}) \rightarrow 0.$$ We also have a natural map $\sha^{1}_{K}(\hat{J}_{n}^{ord}) \rightarrow \sha^{1}_{K}(J_{\infty}^{ord})[\omega_n]$. Note that $\sha^{1}_{K}(J_{\infty}^{ord})$ is a cofinitely generated $\Lambda$-module.

\smallskip

\subsubsection{Cotorsionness of $\Lambda$-adic $\sha^{1}$}
Next we show that under the Hida's control theorem of the Selmer groups (Theorem \ref{Theorem H}) and the finiteness of $\sha_{K}^{1}(\hat{J}_{n}^{ord})$, $\sha^{1}_{K}(J_{\infty}^{ord})^{\vee}$ is a finitely generated torsion $\Lambda$-module. This can be regarded as a $\Lambda$-adic analogue of the Tate-Shafarevich conjecture.

\smallskip

\begin{theorem}\label{theorem 9.2.5} Suppose that $\sha_{K}^{1}(\hat{J}_{n}^{ord})$ is finite for all $n$. Then the functor $\mathfrak{G}$ induces an isomorphism $$\sha^{1}_{K}(J_{\infty}^{ord})^{\vee} \simeq \mathfrak{G}\left(\mathrm{Sel}_{K}(J_{\infty}^{ord})^{\vee}\right)$$ of $\Lambda$-modules. In particular, $\sha^{1}_{K}(J_{\infty}^{ord})^{\vee}$ is a finitely generated torsion $\Lambda$-module.
\end{theorem}

\smallskip

\begin{proof} By taking $p^{\infty}$-torsion part and projective limit of an exact sequence $$0 \rightarrow \sha_{K}^{1}(\hat{J}_{n}^{ord})^{\vee} \rightarrow \mathrm{Sel}_{K}(\hat{J}_{n}^{ord})^{\vee} \rightarrow (\hat{J}^{ord}_{n}(K) \otimes_{\mathbb{Z}_p} \mathbb{Q}_p/\mathbb{Z}_p)^{\vee} \rightarrow 0,$$ we have isomorphisms $\displaystyle  \sha^{1}_{K}(J_{\infty}^{ord})^{\vee} \simeq \lim_{\substack{\longleftarrow \\n}}\sha_{K}^{1}(\hat{J}_{n}^{ord})^{\vee} \simeq \lim_{\substack{\longleftarrow \\n}}\mathrm{Sel}_{K}(\hat{J}_{n}^{ord})^{\vee}[p^{\infty}]$. Hence it is sufficient to show the following isomorphism: $$\displaystyle \lim_{\substack{\longleftarrow \\n}}\mathrm{Sel}_{K}(\hat{J}_{n}^{ord})^{\vee}[p^{\infty}] \simeq \mathfrak{G}\left(\mathrm{Sel}_{K}(J_{\infty}^{ord})^{\vee}\right).$$

\smallskip

By Lemma \ref{lemma A.0.2}-(3), we get an exact sequence $$0 \rightarrow \mathrm{Coker}(s_n)^{\vee} \rightarrow \frac{\mathrm{Sel}_{K}(J_{\infty}^{ord})^{\vee}}{\omega_n \mathrm{Sel}_{K}(J_{\infty}^{ord})^{\vee}}[p^{\infty}]  \rightarrow \mathrm{Sel}_{K}(\hat{J}_{n}^{ord})^{\vee}[p^{\infty}] \rightarrow \mathrm{Ker}(s_n)^{\vee}$$ since $\mathrm{Coker}(s_n)$ is finite for all $n$. Since taking projective limit preserves the exact sequence of compact modules, we have an exact sequence $$\displaystyle 0=\lim_{\substack{\longleftarrow \\n}}\mathrm{Coker}(s_n)^{\vee} \rightarrow \lim_{\substack{\longleftarrow \\n}}\frac{\mathrm{Sel}_{K}(J_{\infty}^{ord})^{\vee}}{\omega_n \mathrm{Sel}_{K}(J_{\infty}^{ord})^{\vee}}[p^{\infty}] \rightarrow \lim_{\substack{\longleftarrow \\n}}\mathrm{Sel}_{K}(\hat{J}_{n}^{ord})^{\vee}[p^{\infty}] \rightarrow \lim_{\substack{\longleftarrow \\n}}\mathrm{Ker}(s_n)^{\vee}=0.$$ Since the second term of the above sequence is $\mathfrak{G}\left(\mathrm{Sel}_{K}(J_{\infty}^{ord})^{\vee}\right)$ by definition, we have the desired assertion.
\end{proof}

\smallskip

\begin{cor} \label{cor 9.2.6} Suppose that $\sha_{K}^{1}(\hat{J}_{n}^{ord})$ and $\sha_{K}^{1}(\hat{J'}_{n}^{ord})$ are finite for all $n$. Then we have $\Lambda$-linear isomorphisms 
\begin{align*}
E\left(\sha^{1}_{K}(J_{\infty}^{ord})^{\vee}\right) &\simeq E\left(\sha^{1}_{K}(J_{\infty}^{' ord})^{\vee}\right)^{\iota} \\
E\left((J_{\infty}^{ord}(K)\otimes_{\mathbb{Z}_p}\mathbb{Q}_p/\mathbb{Z}_p)^{\vee}\right) &\simeq E\left((J_{\infty}^{' ord}(K)\otimes_{\mathbb{Z}_p}\mathbb{Q}_p/\mathbb{Z}_p)^{\vee}\right)^{\iota}
\end{align*}
\end{cor}

\smallskip

\begin{proof} This follows from Theorem \ref{theorem 8.3.10} and Theorem \ref{theorem 9.2.5}.
\end{proof}

\smallskip

\subsubsection{Estimates on the size of $\sha^{1}$} Under the boundedness of $\mathrm{Coker}(s_n)$, we can compute an estimate of the size of $\sha_{K}^{1}(\hat{J}_{n}^{ord})$.

\begin{theorem}\label{theorem 9.2.7} Suppose that $\mathrm{Coker}(s_n)$ is (finite and) bounded independent of $n$, and also suppose that $\sha^{1}_{K}(\hat{J}_{n}^{ord})$ is finite for all $n$. Then there exists an integer $\nu$ such that $$ | \sha^{1}_{K}(\hat{J}_{n}^{ord}) |=p^{e_n} \quad (n>>0)$$ where $$e_n=p^n \mu \left(\sha^{1}_{K}(\hat{J}_{\infty}^{ord})^{\vee} \right) + n \lambda \left(\sha^{1}_{K}(\hat{J}_{\infty}^{ord})^{\vee} \right) +\nu.$$\end{theorem}

\smallskip

\begin{proof} By the straight forward calculation, we can check that for a finitely generated $\Lambda$-module $Y$, we have $$\mathrm{log}_{p} |\frac{Y}{\omega_nY}[p^{\infty}]|=p^n \mu \left(\mathfrak{G}(Y) \right) + n \lambda \left(\mathfrak{G}(Y) \right) + \nu$$ for all $n>>0$.

\medskip

If we let $Y=\mathrm{Sel}_{K}(J_{\infty}^{ord})^{\vee}$, this gives an estimate for the group $\frac{\mathrm{Sel}_{K}(J_{\infty}^{ord})^{\vee}}{\omega_n \mathrm{Sel}_{K}(J_{\infty}^{ord})^{\vee}}[p^{\infty}] $. Since $\mathrm{Ker}(s_n)$ and $\mathrm{Coker}(s_n)$ are bounded independent of $n$, we get the desired assertion since we have isomorphisms $$\mathrm{Sel}_{K}(\hat{J}_n^{ord})^{\vee}[p^{\infty}] \simeq \sha^{1}_{K}(\hat{J}_n^{ord})^{\vee} \quad \mathrm{and} \quad \sha^{1}_{K}(J_{\infty}^{ord})^{\vee} \simeq \mathfrak{G}\left(\mathrm{Sel}_{K}(J_{\infty}^{ord})^{\vee}\right).$$
\end{proof}

\begin{cor}\label{cor 9.2.8} Suppose that $\mathrm{Coker}(s_n), \mathrm{Coker}(s'_n)$ are (finite and) bounded independent of $n$,. If the groups $\sha^{1}_{K}(\hat{J}_{n}^{ord}), \sha^{1}_{K}(\hat{J'}_{n}^{ord})$ are finite for all $n$, then the ratios $$ \frac{|\sha^{1}_{K}(\hat{J}_{n}^{ord})|}{|\sha^{1}_{K}(\hat{J'}_{n}^{ord})|} \quad \mathrm{and} \quad \frac{|\sha^{1}_{K}(\hat{J'}_{n}^{ord})|}{|\sha^{1}_{K}(\hat{J}_{n}^{ord})|} $$ are bounded independent of $n$.
\end{cor}

\begin{proof} This follows from Theorem \ref{theorem 8.3.10}, Theorem \ref{theorem 9.2.5}, and Theorem \ref{theorem 9.2.7}.
\end{proof}

\appendix
\renewcommand{\thetheorem}{\Alph{section}.\arabic{subsection}.\arabic{theorem}}

\section{Lemmas in commutative algebra}

In this appendix, we record some lemmas in commutative algebra.

\begin{lemma} \label{lemma A.0.1}Let $X$ be a finitely generated $\Lambda$-module. Then we have:
(1) $\displaystyle \lim_{\substack{\longleftarrow \\ n}} \frac{X}{\omega_{n}X} \simeq X$. 

(2) $\displaystyle \lim_{\substack{\longleftarrow \\n}}X[\omega_{n}]=0$, where the inverse limit is taken with respect to the norm maps. Hence $\displaystyle \lim_{\substack{\longrightarrow \\ n}} \frac{X^{\vee}}{\omega_{n} X^{\vee}}=0$.

(3) If $X[\omega_{n}]=0$ for all $n$, then $X$ has no non-trivial finite $\Lambda$-submodules and $char(X_{\Lambda-\mathrm{tor}})$ is coprime to $\omega_{n}$ for all $n$. 

(4) $X$ is a $\Lambda$-torsion module whose characteristic ideal is coprime to $\omega_{n}$ for all $n$ if and only if $\frac{X}{\omega_nX}$ is finite for all $n$. In this case, $\displaystyle \lim_{\substack{\longrightarrow \\ n}} X[\omega_{n}] \simeq K$ where $K$ is the maximal finite $\Lambda$-submodule of $X$. Here the direct limit is taken with respect to the natural inclusion.
\end{lemma}

\medskip

\begin{proof} We only prove (2), (3) and (4). For (2), if $X$ is either finite, or free or isomorphic to $\frac{\Lambda}{g^{e}}$ for prime element $g$ which is coprime to $\omega_{n}$ for all $n$, the assertion is obvious. If $X$ is isomorphic to $\frac{\Lambda}{\omega_{m+1, m}^{e}}$ for some $m$ and $e \geq 1$, then $X[\omega_{n}] \simeq \frac{\Lambda}{\omega_{m+1, m}} \simeq \mathbb{Z}_p$ for all $n \geq m+1$, so we have $\displaystyle \lim_{\substack{\longleftarrow \\n}}X[\omega_{n}]=0$. 

Now consider the general $X$. By the structure theorem, we have a exact sequence $$\displaystyle 0 \rightarrow K \rightarrow X \rightarrow S=\Lambda^{r} \oplus (\bigoplus_{i=1}^{t} \frac{\Lambda}{f_{i}^{k_{i}}}) \rightarrow C \rightarrow 0$$ where $K$ and $C$ are finite modules, and $f_{1}, \cdots, f_{t}$ are prime elements of $\Lambda$. By what we have proven so far, $\displaystyle \lim_{\substack{\longleftarrow \\ n}}K[\omega_{n}]=\lim_{\substack{\longleftarrow \\ n}}S[\omega_{n}]=0$, so we get $\displaystyle \lim_{\substack{\longleftarrow \\ n}}X[\omega_{n}]=0$\\

Now we consider (3). Let $K$ be the maximal finite submodule of $X$. For large enough $n$, $K=K[\omega_{n}] \hookrightarrow X[\omega_{n}]$ so we get the first assertion. If $char(X_{\Lambda-\mathrm{tor}})$ contains $\frac{\omega_{n+1}}{\omega_{n}}$ for some $n$, then the structure theorem implies that $X_{\Lambda-\mathrm{tor}}[\omega_{n+1}]$ has $\mathbb{Z}_p$-rank at least $1$, which is a contradiction.\\

We lastly prove (4). The first equivalence follows directly from the structure theorem. Hence we have $K[\omega_{n}]=X[\omega_{n}]$ for all $n$. Since $K$ is finite, $K[\omega_{n}]=K$ for sufficiently large $n$, so the assertion follows.
\end{proof}

\smallskip

\begin{lemma}\label{lemma A.0.2} (1) Let $R$ be an integral domain and $Q(R)$ be the quotient field of $R$. Consider an exact sequence of $R$-modules $0 \rightarrow A \rightarrow B \rightarrow C \rightarrow 0$ where $A$ is a $R$-torsion module. Then we have a short exact sequence $0 \rightarrow A_{R-\mathrm{tor}} \rightarrow B_{R-\mathrm{tor}} \rightarrow C_{R-\mathrm{tor}} \rightarrow 0$.

\medskip

(2) Let $0\rightarrow A \rightarrow B \rightarrow C \rightarrow 0$ be a short exact sequence of finitely generated $\mathbb{Z}_{p}$-modules. If $A$ has finite cardinality, then we have a short exact sequence $0\rightarrow A=A[p^{\infty}] \rightarrow B[p^{\infty}] \rightarrow C[p^{\infty}] \rightarrow 0$.

\medskip

(3) If $0 \rightarrow X \rightarrow Y \rightarrow  Z \rightarrow W \rightarrow 0$ is an exact sequence of cofinitely generated $\mathbb{Z}_p$-modules with finite $W$, then the sequence $  X_{ / div} \rightarrow  Y_{/ div} \rightarrow  Z_{/ div} \rightarrow W \rightarrow 0$ is exact. 
\end{lemma}

\begin{proof} This is \cite[Lemma 2.1.4]{Lee2018}.
\end{proof}

\smallskip

\begin{lemma}
\label{lemma A.0.3}
(1) Let $M$ be a finitely generated $\Lambda$-module, and let $\lbrace \pi_{n}\rbrace$ be a sequence of non-zero elements of $\Lambda$ such that $\pi_0 \in m, \pi_{n+1} \in \pi_{n}m,  \frac{M}{\pi_{n}M}$ is finite for all $n$ where $m$ is the maximal ideal of $\Lambda$. Then we have an isomorphism $$\displaystyle \left(\lim_{\substack{\longrightarrow \\ n}} \frac{M}{\pi_{n} M}\right)^{\vee}   \simeq \mathrm{Ext}^1_{\Lambda}(M, \Lambda)^{\iota}$$ as $\Lambda$-modules.

\medskip

(2) For a finitely generated $\Lambda$-module $X$, we have an isomorphism $\displaystyle \lim_{\substack{\longleftarrow \\ n}}T_{p}(X^{\vee}[\omega_n]) \simeq \mathrm{Hom}_{\Lambda}(X, \Lambda)^{\iota}$.
\end{lemma}

\begin{proof} This is \cite[Lemma A.1.2]{Lee2018}.
\end{proof}

\bibliography{reference}

\begin{thebibliography}{NSWS00}

\bibitem[CS00]{coates2000galois}
John Coates and Ramdorai Sujatha.
\newblock {\em Galois cohomology of elliptic curves}.
\newblock Narosa, 2000.

\bibitem[DDT95]{darmon1995fermat}
Henri Darmon, Fred Diamond, and Richard Taylor.
\newblock Fermat’s last theorem.
\newblock {\em Current Developments in Mathematics}, 1(1):157, 1995.

\bibitem[Fla90]{flach1990generalisation}
Matthias Flach.
\newblock A generalisation of the {Cassels}-{Tate} pairing.
\newblock {\em J. reine angew. Math}, 412:113--127, 1990.

\bibitem[Gre89]{greenberg1989iwasawa}
Ralph Greenberg.
\newblock Iwasawa theory of p-adic representations.
\newblock {\em Algebraic number theory}, pages 97--137, 1989.

\bibitem[GS94]{greenberg1994conjecture}
Ralph Greenberg and Glenn Stevens.
\newblock On the conjecture of {Mazur}, {Tate}, and {Teitelbaum}.
\newblock {\em Contemporary Mathematics}, 165:183--183, 1994.

\bibitem[GV00]{greenberg2000iwasawa}
Ralph Greenberg and Vinayak Vatsal.
\newblock On the {Iwasawa} invariants of elliptic curves.
\newblock {\em Inventiones mathematicae}, 142(1):17--63, 2000.

\bibitem[Hid86]{hida1986galois}
Haruzo Hida.
\newblock Galois representations into
  $\text{{G}{L}}_{2}(\mathbb{Z}_{p}[\![{X}]\!])$ attached to ordinary cusp
  forms.
\newblock {\em Inventiones mathematicae}, 85(3):545--613, 1986.

\bibitem[Hid12]{hida2012p}
Haruzo Hida.
\newblock {\em p-adic automorphic forms on Shimura varieties}.
\newblock Springer Science \& Business Media, 2012.

\bibitem[Hid13a]{hida2013lambda}
Haruzo Hida.
\newblock {$\Lambda$}-adic {Barsotti}--{Tate} groups.
\newblock 2013.

\bibitem[Hid13b]{hida2013limit}
Haruzo Hida.
\newblock Limit {Mordell}--{Weil} groups and their $p$-adic closure.
\newblock 2013.

\bibitem[Hid17]{hida2015analytic}
Haruzo Hida.
\newblock Analytic variation of {Tate}-{Shafarevich} groups.
\newblock preprint, 2017.

\bibitem[Lee18]{Lee2018}
Jaehoon Lee.
\newblock Structure of {Mordell}-{Weil} groups over $\mathbb{Z}_p$-extensions.
\newblock {\em \href{https://arxiv.org/abs/1809.10351}{arXiv:1809.10351}},
  2018.

\bibitem[Maz72]{mazur1972rational}
Barry Mazur.
\newblock Rational points of abelian varieties with values in towers of number
  fields.
\newblock {\em Inventiones mathematicae}, 18(3-4):183--266, 1972.

\bibitem[Mil06]{milne2006arithmetic}
James~S Milne.
\newblock {\em Arithmetic duality theorems}.
\newblock Citeseer, 2006.

\bibitem[MTT86]{mazur1986onp}
Barry Mazur, John Tate, and Jeremy Teitelbaum.
\newblock On $p$-adic analogues of the conjectures of {Birch} and
  {Swinnerton-Dyer}.
\newblock {\em Inventiones mathematicae}, 84(1):1--48, 1986.

\bibitem[NSWS00]{neukirch2000cohomology}
J{\"u}rgen Neukirch, Alexander Schmidt, Kay Wingberg, and Alexander Schmidt.
\newblock {\em Cohomology of number fields}, volume 323.
\newblock Wiley Online Library, 2000.

\end{thebibliography}
\bibliographystyle{alpha}

\end{document}